\documentclass[12pt,a4paper,oneside,reqno]{amsproc}
\usepackage{amsmath} 
\usepackage{amssymb,amsfonts,mathrsfs}
\usepackage[colorlinks=true,linkcolor=blue,citecolor=blue]{hyperref}
\usepackage[driver=pdftex,lmargin=2.5cm,rmargin=2.5cm,heightrounded=true,centering]{geometry}
\usepackage{math50}
\usepackage{local}
\usepackage{epsfig}  
\usepackage{graphicx}  
\usepackage{xy}
\xyoption{all}

\begin{document}
\title[Fractional Schr\"odinger equations with singular potential]
{Fractional nonlinear Schr\"odinger equations with singular potential in $\mathbf R^n$}
\author{Guoyuan Chen}
\address{School of Mathematics and Statistics, Zhejiang University of Finance \& Economics, Hangzhou 310018, Zhejiang, P. R. China}
\email{gychen@zufe.edu.cn}
\author{Youquan Zheng}
\address{School of Science, Tianjin University, Tianjin 300072, P. R. China.}
\email{zhengyq@tju.edu.cn}

\newcommand{\optional}[1]{\relax}
\setcounter{secnumdepth}{3}
\setcounter{section}{0} \setcounter{equation}{0}
\numberwithin{equation}{section}


\thanks{G. Chen was partially supported by National Natural Science Foundation of China (No.11401521 and 11402226). 
Y. Zheng was partially supported by National Natural Science Foundation of China (No.11301374) and SRFDP (20130032120077)}
\date{\today}
\begin{abstract}
We are interested in nonlinear fractional Schr\"odinger equations with singular potential of form
\begin{equation*}
(-\Delta)^su=\frac{\lambda}{|x|^{\alpha}}u+|u|^{p-1}u,\quad \mathbf R^n\setminus\{0\},
\end{equation*}
where $s\in (0,1)$, $\alpha>0$, $p\ge1$ and $\lambda\in \mathbf R$. Via Caffarelli-Silvestre extension method, we obtain existence, nonexistence, regularity and symmetry properties of solutions to this equation for various $\alpha$, $p$ and $\lambda$.
\end{abstract}
\maketitle

\section{Introduction and main results}

The purpose of this paper is to investigate nonlinear fractional Schr\"odinger equations with singular potentials as follows,
\begin{equation}\label{e:main1}
(-\Delta)^su=\frac{\lambda}{|x|^{\alpha}}u+|u|^{p-1}u,\quad \mathbf R^n\setminus\{0\},
\end{equation}
where $(-\Delta)^s$ is the fractional Laplacian with $s\in (0,1)$, $\alpha>0$, $p\geq 1$ and $\lambda\in \mathbf R$ .

Because of the nonlocal property, traditional analysis techniques for local differential equations become difficult when dealing with (\ref{e:main1}). Instead, we use the extension method for fractional Laplacian developed by Caffarelli and Silvestre in \cite{CaSiCPDE2007}.
Let $\mathbf R^{n+1}_+=\mathbf R^n\times\mathbf R_+$, then $\partial\mathbf R^{n+1}_+=\mathbf R^n$.
We consider the following degenerate elliptic problem
\begin{equation}\label{e:mail-extention}
\left\{\begin{array}{rl}
         {\rm div}(t^{1-2s}\nabla U)=0, & \mbox{ in }\mathbf R^{n+1}_+,\\
         -\displaystyle\lim_{t\to 0^+}t^{1-2s}\partial_t U(x,t)=\kappa_s\left(\dfrac{\lambda}{|x|^{\alpha}}U(x,0)+(|U|^{p-1}U)(x,0)\right), & \mbox{ on }\partial\mathbf R^{n+1}_+,
       \end{array}
\right.
\end{equation}
where $\kappa_s=\frac{\Gamma(1-s)}{2^{2s-1}\Gamma(s)}$. By the extension formula in \cite{CaSiCPDE2007}, the trace of a solution $U(x, t)$ for (\ref{e:mail-extention}) on $\partial\mathbf R^{n+1}_+$, $U(x,0)$ is a solution of (\ref{e:main1}).
Let
\begin{equation*}
H^1(\mathbf R^{n+1}_+,t^{1-2s}):=\left\{U\,\Big|\,\int_{\mathbf R^{n+1}_+}(U^2+|\nabla U|^2)t^{1-2s}dxdt<\infty\right\}
\end{equation*}
with inner product
\begin{equation*}
\langle U,V\rangle_{H^1(\mathbf R^{n+1}_+,t^{1-2s})}=\int_{\mathbf R^{n+1}_+}(UV+\nabla U\cdot\nabla V)t^{1-2s}dxdt, \quad U,\,V\in H^1(\mathbf R^{n+1}_+,t^{1-2s}),
\end{equation*}
and
$$L^2(\mathbf R^n,\frac{1}{|x|^{\alpha}}):=\left\{u\,\Big|\,\int_{\mathbf R^n}\frac{|u|^2}{|x|^{\alpha}}dx<\infty\right\}.$$
Let $\mathcal B_r$ (resp. $B_r$) be the ball with radius $r$ centered at $0$ in $\mathbf R^{n+1}$ (resp. $\mathbf R^n$), $\mathcal B_r^+:=\mathcal B_r\cap \mathbf R^{n+1}_+$ and $\mathbf S_r^+:=\partial \mathcal B_r\cap \mathbf R^{n+1}_+$.

The first result of this article is a Pohozaev type identity for (\ref{e:mail-extention}). As its application, we then obtain some non-existence results for (\ref{e:main1}). For concrete problems involving fractional Laplacian, many authors have obtained various Pohozaev type identities. We refer the interested readers to \cite{cabre2010positive,brandle2013concave,cabre2014sharp,FallFelliCPDE2014,RS:ARMA14,YanYangYuJFA2015} and the references therein for more related results.

Our Pohozaev type identity reads as follows.
\begin{theorem}\label{t:pozaev-identity}
If $U\in H^1(\mathbf R^{n+1}_+,t^{1-2s})$ with $U(\cdot, 0)\in L^{p+1}(\mathbf R^n)\cap L^2(\mathbf R^n,\frac{1}{|x|^{\alpha}})$ is a weak solution to (\ref{e:mail-extention}) (see Definition \ref{d:weak-solutions} below for the definition of weak solutions), then for almost everywhere $r\in (0,+\infty)$, it holds that
\begin{eqnarray}\label{e:pohozaev0}
&&\kappa_s\lambda\left(\frac{2s-\alpha}{2}\right)\int_{B_r}\frac{U^2}{|x|^{\alpha}}+\kappa_s\left(\frac{n}{p+1}-\frac{n-2s}{2}\right)\int_{B_r}|U|^{p+1}\\
&&=\frac{r}{2}\int_{\mathbf S_r^+}t^{1-2s}|\nabla U|^2-r\int_{\mathbf S_r^+}t^{1-2s}\left|\frac{\partial U}{\partial\nu}\right|^2
-\frac{\kappa_s\lambda r}{2}\int_{\partial B_r}\frac{U^2}{|x|^{\alpha}}\notag\\
&&\quad -\frac{\kappa_sr}{p+1}\int_{\partial B_r}|U|^{p+1}-\frac{n-2s}{2}\int_{\mathbf S^+_r}t^{1-2s}\frac{\partial U}{\partial\nu}U,\notag
\end{eqnarray}
here $\nu$ is the unit out normal vector of $\mathbf S_r^+$.
\end{theorem}

Let $2^*(s)=\frac{2n}{n-2s}$ and $H^s(\mathbf R^n)$ be the Sobolev space of order $s$. As applications of the Pohozaev type identity above, we have
\begin{cor}\label{c:nonexistence-g}
\begin{enumerate}
  \item[(1)] If $\lambda\ge 0$, $\alpha<2s$ and $1\le p<2^*(s)-1$, then (\ref{e:main1}) has no nontrivial solution in $ H^{s}(\mathbf R^{n})\cap L^2(\mathbf R^n,\frac{1}{|x|^{\alpha}})$;
  \item[(2)] If $\alpha=2s$ and $p\ne 2^*(s)-1$, then there exists no nontrivial solution to (\ref{e:main1}) in $H^{s}(\mathbf R^{n})\cap L^{p+1}(\mathbf R^n)$;
  \item[(3)] If $\alpha\ne 2s$ and $p= 2^*(s)-1$, then there exists no nontrivial solution to (\ref{e:main1}) in $H^{s}(\mathbf R^{n})\cap L^{2}(\mathbf R^n,\frac{1}{|x|^{\alpha}})$.
\end{enumerate}
\end{cor}

Next, we focus on the nonnegative solutions to (\ref{e:main1}) for $\alpha=2s$ and $p=2^*(s)-1$. In this case, (\ref{e:main1}) becomes
\begin{equation}\label{e:main3}
(-\Delta)^su=\frac{\lambda}{|x|^{2s}}u+u^{\frac{n+2s}{n-2s}},\quad \mathbf R^n\setminus\{0\}.
\end{equation}

The operator $H:=(-\Delta)^s-\frac{\lambda}{|x|^{2s}}$ appears in the study of stability of relativistic matter in magnetic fields (see i.e. \cite{HerbstCMP1977,YORJMP1997,FLS:JAMS2007}). It is also related to the following Hardy-type inequality which was obtained by Herbst \cite{HerbstCMP1977} (see also \cite{Becker:PAMS95,Yafaev:JFA99}):
\begin{equation}\label{e:hardy-ineq}
\Lambda_{n,s}\int_{\mathbf R^n}|x|^{-2s}|u(x)|^2dx\le \int_{\mathbf R^n}|\xi|^{2s}|\hat u(\xi)|^2d\xi,\quad u\in C_0^{\infty}(\mathbf R^n),
\end{equation}
where the sharp constant $\Lambda_{n,s}$ is given by $2^{2s}\frac{\Gamma^2((n+2s)/4)}{\Gamma^2((n-2s)/4)}$, $\Gamma$ is the usual gamma function and $\hat u$ denotes the Fourier transform of $u$. Note that $\Lambda_{n,s}$ converges to the classical Hardy constant $(n-2)^2/4$ as $s\to 1$. Some basic properties (Hardy-Lieb-Thirring inequalities, self-adjointness, spectrum, unique continuation, etc.) of such kind of relativistic Schr\"odinger operator have been investigated. See, for example, \cite{FLS:JAMS2007}, \cite{Frank:CMP2009}, \cite{FallJFA2014}, \cite{FallFelliCPDE2014}. When $s=1$, we refer the readers to \cite{ReedSimon75,Kato95} and the references therein for some related results.

In recent years, motivated by various applications, there are large amount of work on fractional Schr\"odinger equations with critical exponent. In \cite{LiYanYanJEMS2004} and \cite{ChenLiOUCPAM2006}, the authors proved that every positive regular entire solution $u(x)$ of equation
\begin{equation}\label{e:fyp}
(-\Delta)^su=u^{\frac{n+2s}{n-2s}}
\end{equation}
is radially symmetric and decreasing about some point $x_0$, precisely,
$$u(x)=c\left(\frac{a}{a^2+|x-x_0|^2}\right)^{\frac{n-2s}{2}}$$
with some positive constants $c$ and $a$. In \cite{CJSXARMA2014}, the authors proved that all nonnegative solutions with isolated singularities to Equation (\ref{e:fyp}) on a ball
are asymptotically radial symmetric,
those results generalize the classical results obtained by \cite{GNN:CMP79}, \cite{CGSCPAM1989}. In \cite{DDGW2015}, Delaunay-type solutions for (\ref{e:fyp}) with an isolated singularity were constructed. Note that (\ref{e:fyp}) is the spacial case of (\ref{e:main3}) when $\lambda=0$, it is closely related to the fractional Yamabe problem, we refer the interested readers to \cite{JinLiXiongJEMS2014, KMW2015} and the references therein on this topic. Consider the extension form of
(\ref{e:main3}) for nonnegative $U$,
\begin{equation}\label{e:mail-extention-c}
\left\{\begin{array}{rl}
         {\rm div}(t^{1-2s}\nabla U)=0, & \mbox{ in }\mathbf R^{n+1}_+,\\
         -\displaystyle\lim_{t\to 0^+}t^{1-2s}\partial_t U(x,t)=\kappa_s\left(\dfrac{\lambda}{|x|^{2s}}U(x,0)+U^{\frac{n+2s}{n-2s}}(x,0)\right), & \mbox{ on }\partial\mathbf R^{n+1}_+.
       \end{array}
\right.
\end{equation}

Then we have
\begin{theorem}\label{t:regulatiry}
Let $\lambda\ge 0$, if $U$ is a nonnegative weak solution to (\ref{e:mail-extention-c}), then (1) $U$ is positive and $U\in C^{\infty}(\overline{\mathbf R^{n+1}_+}\setminus\{0\})$; (2) for any $t\ge 0$, $U(x,t)$ is radial symmetric, decreasing in radial directions with respect to $x$.
\end{theorem}

The proof of Theorem \ref{t:regulatiry} is based on a combination of the Cafferelli-Silvestre's extension and the method of moving spheres.
The method of moving planes or moving spheres is a strong technique in studying the symmetry and monotonicity of solutions of various elliptic differential equations and some conformal invariant integral equations. See, for example, \cite{Serrin:ARMA71,GNN:CMP79,CGSCPAM1989,BerestyckiNirenberg1991,ChenLi:Duke91, LiZhuDukeMath1995,Li:Invent96,LiYanYanJEMS2004,LiYanYan2006,ChenLiOUCPAM2006,JinQinianLiyanyanXuADE2008} and the references therein.
In \cite{JinLiXiongJEMS2014,CJSXARMA2014}, the authors applied the method of moving spheres to fractional Yamabe equations successfully by using the Cafferelli-Silvestre's extension technique. In \cite{ChenLiLi2014, barrios2014moving, JarohsWethDCDS2014, jarohs2014symmetry, felmer2014radial, DiMoPeSc15}, the authors developed a direct method of moving planes to prove the symmetry and existence of solutions to some semi-linear elliptic equations involving fractional Laplacian. Both of these approaches need to recover concrete maximum principles as in the classical case.

When using the method of moving spheres (or moving planes), each problem has its own difficulties. In our problem, equation (\ref{e:mail-extention-c}) is not exactly conformally invariant as the ones in \cite{JinLiXiongJEMS2014,CJSXARMA2014} ($\lambda=0$), that is, after Kelvin transformation, (\ref{e:main3}) becomes
\begin{equation}\label{e:ktfl}
(-\Delta)^su_{x_0,\rho}-\left(\frac{\rho}{|x-x_0|}\right)^{4s}\frac{\lambda}{|x_{\rho,x_0}|^{2s}}u_{x_0,\rho}
=u_{x_0,\rho}^{\frac{n+2s}{n-2s}},
\end{equation}
where $x_0$ is any fixed point in $\mathbf R^n\setminus\{0\}$, $u_{x_0,\rho}$ is the Kelvin transformation of $u$:
\begin{equation*}
u_{x_0,\rho}(\xi)=\left(\frac{\rho}{|\xi-x_0|}\right)^{n-2s}u\left(x_0+\frac{\rho^2(\xi-x_0)}{|\xi-x_0|^2}\right),\quad \xi\in \mathbf R^{n}\setminus \{x_0\}.
\end{equation*}
For the detailed proof of this equation, see Appendix \ref{s:kt} below. The second term in the left side of (\ref{e:ktfl}) makes the verification of main ingredients of the method of moving spheres more complicated. For more details, see Section \ref{s:symmetry}.

As a consequence of Theorem \ref{t:regulatiry}, we have
\begin{cor}\label{c:fractional-symmetry}
Let $\lambda\ge 0$, if $u$ is a nonnegative weak solution to (\ref{e:main3}), then $u$ is smooth and positive in $\mathbf R^{n}\setminus\{0\}$, radial symmetric about the origin and strictly decreasing in radial directions.
\end{cor}
\begin{remark}
When $s=1$, this kind of results was obtained in \cite{Terrachini:ADE96,JinQinianLiyanyanXuADE2008}.
\end{remark}

In the classical case $s=1$, when $\lambda$ is greater than the Hardy constant $(n-2)^2/4$, there is no positive solution to (\ref{e:main3}), see, for example, \cite{Terrachini:ADE96,JinQinianLiyanyanXuADE2008}. In the fractional case, we have a similar nonexistence result as follows.
\begin{theorem}\label{t:lglne}
Assume that  $\lambda\ge \Lambda_{n,s}$, there is no nonnegative nontrivial solution of (\ref{e:mail-extention-c}) in $\dot{H}^1(\mathbf R^{n+1}_+,t^{1-2s})$.
\end{theorem}
Here $\dot{H}^1(\mathbf R^{n+1}_+,t^{1-2s})$ denotes the completion of the set $C^{\infty}_0(\overline{\mathbf R^{n+1}_+})$ under the norm
\begin{equation*}
\|U\|_{\dot{H}^1(\mathbf R^{n+1}_+,t^{1-2s})}=\left(\int_{\mathbf R^{n+1}_+}t^{1-2s}|\nabla U|\right)^{\frac{1}{2}}.
\end{equation*}

The key point of verifying this theorem is constructing proper test functions by an eigenvalue problem on the up-half unit sphere in $\mathbf R^{n+1}_+$. To be more precisely, let $\mathbf S^n$ be the unit $n$-dimensional sphere in $\mathbf R^{n+1}$ and
\begin{equation*}
\mathbf S^n_+=\{\theta=(\theta_1,\cdots,\theta_n,\theta_{n+1})\in \mathbf S^n:\,\theta_{n+1}>0\}.
\end{equation*}
From the compactness of $\mathbf S^n_+$ and \cite[Theorem 6.16]{Salsa08},
there exists some positive function $\psi_1$ in $H^1(\mathbf S^n_+,\theta_{n+1}^{1-2s})$ satisfying
\begin{equation}\label{e:sextension}
\left\{\begin{array}{cc}
         {\rm div}_{\mathbf S^n}(\theta_{n+1}^{1-2s}\nabla_{\mathbf S^n}\psi_1)=\left(\frac{2s-n}{2}\right)^2\theta_{n+1}^{1-2s}\psi_1 & \mbox{ in }\mathbf S^n_+,\\
          -\displaystyle\lim_{\theta_{n+1}\to 0}\theta_{n+1}^{1-2s}\nabla_{\mathbf S^n}\psi_1\cdot \mathbf e_{n+1}=\Lambda_{n,s}\kappa_s\psi_1, & \mbox{ on }\partial\mathbf S^n_+.
       \end{array}
\right.
\end{equation}
Here $\mathbf e_{n + 1} = (0,\cdot\cdot\cdot, 0, 1)\in \mathbf S^n_+$.
This yields the nonexistence for $\lambda\ge \Lambda_{n,s}$ by choosing $|X|^{\frac{2s-n}{2}}\psi_1\left(\frac{X}{|X|}\right)$ as a test function for problem (\ref{e:mail-extention-c}). For more details, see Section \ref{sb:proof-1.5} below.

As a consequence, we have
\begin{cor}
Assume that $\lambda\ge \Lambda_{n,s}$, then there is no nonnegative nontrivial solution of (\ref{e:main3}) in $\dot{H}^s(\mathbf R^n)$.
\end{cor}
Here and what follows, $\dot{H}^s(\mathbf R^n)$ denotes the completion of $C_0^{\infty}(\mathbf R^n)$ with respect to the norm $\|\,\cdot\,\|_{\dot{H}^s(\mathbf R^n)}$ induced by the scalar product
\begin{equation}\label{e:inner-product-f}
\int_{\mathbf R^n}|\xi|^{2s}\hat u(\xi)\hat v(\xi)d\xi,\quad u,\,v\in C_0^{\infty}(\mathbf R^n).
\end{equation}
We denote the corresponding inner product by $\langle\cdot ,\cdot\rangle_{\dot{H}^s(\mathbf R^n)}$.

Finally, we will prove the following existence result.
\begin{theorem}\label{t:existence-6}
If $0\le \lambda<\Lambda_{n,s}$, then there exists a positive solution to (\ref{e:mail-extention-c}) in $\dot{H}^1(\mathbf R^{n+1}_+,t^{1-2s})$.
\end{theorem}
In order to prove this result, it is natural to analyze the corresponding minimization problem, see (\ref{e:minimization}).
Firstly, we prove that any minimization sequence $\{U_k\}$ with nontrivial weak limit $V$ yields that $\{U_k\}$ approaches to $V$ with respect to the $H^1(\mathbf R^{n+1}_+,t^{1-2s})$-norm topology. Thus the existence follows. Secondly, inspired by \cite[Proof of Theorem 1.5]{DiMoPeSc15}, we show that if a minimization sequence $\{U_k\}$ has a trivial weak limit, then after rearrangement and some proper rescaling of $\{U_k\}$, there exists a minimization sequence $\{\hat U_k\}$ with nontrivial weak limit.

From Corollary \ref{c:fractional-symmetry} and Theorem \ref{t:existence-6}, we have

\begin{cor}\label{c:f-exitence-s}
If $0\le \lambda< \Lambda_{n,s}$, then (\ref{e:main3}) has a positive solution in $\dot{H}^s(\mathbf R^n)$ and it is smooth and positive in $\mathbf R^{n}\setminus\{0\}$, radial symmetric about the origin and strictly decreasing in radial directions.
\end{cor}
\begin{remark}
(1) For the classical case $s=1$, the existence and classification of positive solutions were investigated in many papers, see e.g. \cite{Terrachini:ADE96}, \cite{JinQinianLiyanyanXuADE2008}, \cite{smets2005nonlinear} and the references therein.

(2) The result in Corollary \ref{c:f-exitence-s} was also proved in \cite{DiMoPeSc15} by a nonlocal version of moving plane method. Moreover, it should be pointed out that the asymptotic property of the solutions to (\ref{e:main3}) for $0\le \lambda< \Lambda_{n,s}$ was obtained in \cite{DiMoPeSc15}. To be more precisely, if $u$ is a positive solution to (\ref{e:main3}), then there exist positive numbers $C_1$ and $C_2$ such that
\begin{equation*}
C_1\left(|x|^{1-\eta}(1+|x|^{2\eta})\right)^{-\frac{n-2s}{2}}\le u(x)\le C_2\left(|x|^{1-\eta}(1+|x|^{2\eta})\right)^{-\frac{n-2s}{2}},
\end{equation*}
where $\eta=1-\frac{2\vartheta}{n-2s}$ for some explicit determined number $\vartheta\in(0,\frac{n-2s}{2})$. For details, see \cite[Theorem 1.7]{DiMoPeSc15}.
\end{remark}

This paper is organized as follows: in Section \ref{s:preliminaries}, we give some preliminaries for our further investigation. In Section \ref{s:pohozaev}, we prove Theorem \ref{t:pozaev-identity} and Corollary \ref{c:nonexistence-g}. Section \ref{s:regularity} contains the proof of regularity property in Theorem \ref{t:regulatiry}. In Section \ref{s:symmetry}, we prove the symmetry and monotonicity properties in Theorem \ref{t:regulatiry} by the method of moving spheres. Sections \ref{s:none} and \ref{s:existence-d} are devoted to the proof of Theorems \ref{t:lglne} and \ref{t:existence-6} respectively. Appendix contains the detailed computation of Kelvin transformation.


\section{Preliminaries}\label{s:preliminaries}

In this section, we illustrate some definitions and basic results.

Let $\Omega$ be a domain in $\mathbf R^n$ with Lipschitz boundary and $s\in (0, 1)$. The $s$-order Sobolev space $H^s(\Omega)$ is defined by
\begin{equation*}
H^s(\Omega):=\left\{u\in L^2(\Omega)\,\Big|\,\int_{\Omega}\int_{\Omega}\frac{|u(x)-u(y)|^2}{|x-y|^{n+2s}}dxdy<\infty\right\}
\end{equation*}
with its norm given by
\begin{equation*}
\|u\|_{H^s{(\Omega)}}:=\left(\int_{\Omega}u^2dx+\int_{\Omega}\int_{\Omega}\frac{|u(x)-u(y)|^2}{|x-y|^{n+2s}}dxdy\right)^{\frac{1}{2}}.
\end{equation*}

Let $X := (x, t)\in \mathbf{R}^{n + 1}$ where $x\in \mathbf{R}^n$ and $t\in \mathbf{R}$, $D\subset \mathbf R^{n+1}_+$ be a domain with Lipschitz boundary. It is well-known that $|t|^{1-2s}$ is an element of the Muckenhoupt $A_2$ class in $\mathbf{R}^{n + 1}$ (see \cite{MuckTAMS72}). Define $L^2(D,t^{1-2s})$ to be the Hilbert space of all measurable functions $U$ on $D$ with norm
\begin{equation*}
\|U\|_{L^2(D,t^{1-2s})}:=\left(\int_{D}t^{1-2s}U^2(X)dX\right)^{\frac{1}{2}} < \infty.
\end{equation*}
The space of all functions $U\in L^2(D,t^{1-2s})$ with its weak derivatives $\nabla U$ exists and belongs to $L^2(D,t^{1-2s})$ is denoted as $H^1(D,t^{1-2s})$. The norm in $H^1(D,t^{1-2s})$ is
\begin{equation*}
\|U\|_{H^1(D,t^{1-2s})}:=\left(\int_{D}t^{1-2s}U^2(X)dX+\int_{D}t^{1-2s}|\nabla U|^2(X)dX\right)^{\frac{1}{2}}.
\end{equation*}
Correspondingly, the inner product on $H^1(D,t^{1-2s})$ is given by
\begin{equation*}
\langle U,V\rangle_{H^1(D,t^{1-2s})}:=\int_{D}t^{1-2s}(UV+\nabla U\cdot\nabla V)dX.
\end{equation*}

Every element in $H^1(D,t^{1-2s})$ has a well-defined trace. To be more precisely, one has the following proposition (for the details of its proof and other related results, see e.g. \cite{FabesKenigCarlosCPDE1982,JinLiXiongJEMS2014}).
\begin{lemma}\label{p:trace}(\cite[Proposition 2.1]{JinLiXiongJEMS2014})
Let $\Omega\subset \mathbf R^n$ be a domain with Lipschitz boundary and $R>0$. Set $D=\Omega\times (0,R)\subset \mathbf R^n\times \mathbf R_+$.
If $U\in H^1(D,t^{1-2s})\cap C(D\cup\partial'D)$, then $u:=U(\,\cdot\,,0)\in H^s(\Omega)$ and
\begin{equation*}
\|u\|_{H^s(\Omega)}\le C\|U\|_{H^1(D,t^{1-2s})},
\end{equation*}
where $C$ is a positive constant only depending on $n$, $s$, $R$ and $\Omega$. Therefore, for every $U\in H^1(D,t^{1-2s})$, the trace $U(\,\cdot\,,0)$ is well-defined and belongs to $H^s(\Omega)$. Furthermore, there exists a positive constant $C_{n,s}$ depending only on $n,\,s$ such that
\begin{equation*}
\|U(\,\cdot\,,0)\|_{L^{\frac{2n}{n-2s}}(\Omega)}\le C_{n,s}\|\nabla U\|_{L^2(D,t^{1-2s})}\quad \mbox{ for all }\quad  U\in C^{\infty}_c(D\cup\partial'D).
\end{equation*}
\end{lemma}
Here and after, for a domain $D\subset \mathbf{R}^{n+1}_+$ with boundary $\partial D$, $\partial'D$ denotes the interior of $\overline{D}\cap \partial \mathbf{R}^{n + 1}_{+}$ in $\mathbf{R}^n = \partial\mathbf{R}^{n + 1}$ and $\partial''D =: \partial D\setminus \partial'D$.

For further applications, we need the following Sobolev embedding inequality which was proved in \cite[Proposition 2.1.1]{DiMeVa15}. The embedding for more general weighted functions of $A_p$ type can be found in \cite{FabesKenigCarlosCPDE1982}.
\begin{lemma}\label{l:ewss}
There exists a constant $C$ such that, for all $U\in \dot{H}^{1}(\mathbf R^{n+1}_+,t^{1-2s})$, one has
\begin{equation*}
\left(\int_{\mathbf R^{n+1}_+}t^{1-2s}|U|^{2\gamma}\right)^{\frac{1}{2\gamma}}\le C\left(\int_{\mathbf R^{n+1}_+}t^{1-2s}|\nabla U|^{2}\right)^{\frac{1}{2}},
\end{equation*}
where $\gamma=1+\frac{2}{n-2s}$.
\end{lemma}

Let $\mathcal{B}_R(X)$ be the ball in $\mathbf{R}^{n + 1}$ with radius $R$ and center $X$, $\mathcal{B}^+_R(X) = \mathcal{B}_R(X)\cap \mathbf{R}^{n + 1}_+$ and $B_R(x)$ be the ball in $\mathbf{R}^n$ with radius $R$ centered at $x$. For simplicity, we will write $\mathcal{B}(0)$, $\mathcal{B}^+(0)$ and $B_R(0)$ as $\mathcal{B}$, $\mathcal{B}^+$ and $B_R$, respectively.
\begin{dfn}\label{d:weak-solutions}
We say $W\in H^1(\mathcal B_R^+,t^{1-2s})\cap L^2(\mathcal B^+_R,\frac{1}{|x|^\alpha})$ is a weak solution to
\begin{equation}\label{e:extention-bounded-domain}
\left\{\begin{array}{rl}
         {\rm div}(t^{1-2s}\nabla U)=0, & \mbox{ in }\mathcal B_R^+,\\
         -\displaystyle\lim_{t\to 0^+}t^{1-2s}\partial_t U(x,t)=\kappa_s\left(\dfrac{\lambda}{|x|^{\alpha}}U(x,0)+(|U|^{p-1}U)(x,0)\right), & \mbox{ on }\partial\mathcal B_R^+=B_R,
       \end{array}
\right.
\end{equation}
if for all $V\in C_0^{\infty}(\mathcal B_R^+\cup B_R)$, the following equation holds,
\begin{equation*}
\int_{\mathbf R^{n+1}_+}t^{1-2s}\nabla W\cdot\nabla V dxdt=\kappa_s\int_{B_R}\left(\dfrac{\lambda}{|x|^{\alpha}}W(x,0)+(|W|^{p-1}W)(x,0)\right)V dx.
\end{equation*}
\end{dfn}

The functions in $\dot{H}^s(\mathbf R^n)$ have natural extension in $\dot{H}^1(\mathbf R^{n+1}_+,t^{1-2s})$. To be more precisely,
for all $u\in \dot{H}^s(\mathbf R^n)$, define
\begin{equation}\label{e:extension-u}
U(x,t)=(\mathcal P_s\ast u)(x,t):=\int_{\mathbf R^n}\mathcal P_s(x-\xi,t)u(\xi)d\xi,\quad (x,t)\in\mathbf R^{n+1}_+,
\end{equation}
where
\begin{equation}\label{e:psk}
\mathcal P_s(x,t)=\iota(n,s)\frac{t^{2s}}{(|x|^2+t^2)^{\frac{n+2s}{2}}}
\end{equation}
with the normalized constant $\iota(n,s)$ such that $\int_{\mathbf R^n}\mathcal P_s(x,1)dx=1$. Then $U\in C^{\infty}(\mathbf R^{n+1}_+)$ and $U\in L^2(K,t^{1-2s})$ for any compact set $K$ in $\overline{\mathbf R^{n+1}_+}$, $\nabla U\in L^2(\mathbf R^{n+1}_+,t^{1-2s})$. Moreover,
\begin{equation*}
{\rm div}(t^{1-2s}\nabla U)=0,\quad \mbox{ in }\mathbf R^{n+1}_+,
\end{equation*}
and
\begin{equation}\label{e:gtrace}
\|\nabla U\|_{L^2(t^{1-2s},\mathbf R^{n+1}_+)}=\kappa_s\|u\|_{\dot{H}^s(\mathbf R^n)}.
\end{equation}
We call $U(x,t)=(\mathcal P_s\ast u)(x,t)$ the extension of $u(x)$ for any $u\in \dot{H}^s(\mathbf R^n)$.

\begin{lemma}\cite[Lemma A.3]{JinLiXiongJEMS2014}\label{l:minimal-eu}
Let $u\in C^{\infty}_0(\mathbf R^n)$ and $U(\cdot,t)=\mathcal P_s(\cdot,t)\ast u(\cdot)$. Then for any $W\in C^{\infty}_0(\overline{\mathbf R^{n+1}_+})$ with
$W(\cdot,0)=u(x)$, it holds that
\begin{equation*}
\int_{\mathbf R^{n+1}_+}t^{1-2s}|\nabla U|^2\le\int_{\mathbf R^{n+1}_+}t^{1-2s}|\nabla W|^2.
\end{equation*}
\end{lemma}
\begin{lemma}\label{l:eminimal}
Let $U\in \dot{H}^1(\mathbf R^{n+1}_+,t^{1-2s})$. Then
\begin{equation}\label{e:trace-embedding}
\kappa_s\|U(\cdot, 0)\|_{\dot{H}^s(\mathbf R^n)}\le \|U\|_{\dot{H}^1(\mathbf R^{n+1}_+,t^{1-2s})}.
\end{equation}
\end{lemma}
\begin{proof}
Combining (\ref{e:gtrace}) and Lemma \ref{l:minimal-eu}, we immediately have (\ref{e:trace-embedding}).
\end{proof}


\section{Proof of Theorem \ref{t:pozaev-identity} and Corollary \ref{c:nonexistence-g}}\label{s:pohozaev}

In this section, we prove a Pohozaev-type identity for (\ref{e:mail-extention}) inspired by the proof of \cite[Theorem 3.1]{FallFelliCPDE2014}.
Moreover, as its applications, we investigate the nonexistence of solutions for problem (\ref{e:main1}).
\begin{proof}[Proof of Theorem \ref{t:pozaev-identity}]
Let $\rho<r<R$. Set $O_{\delta}:=(\mathcal B_r^+\setminus \overline{\mathcal B_{\rho}^+})\cap \{(x,t)\,|\, t>\delta\}$ with $\delta>0$.
Let $\partial' O_{\delta}=\bar O_{\delta}\cap \{t=\delta\}$, $\partial''O_{\delta}=\partial O_{\delta}\setminus \partial'O_{\delta}$ and $\nu$ be the unit outer normal of $\partial O_{\delta}$. Multiplying (\ref{e:mail-extention}) by $X\cdot\nabla U$ and integrating by parts in $O_{\delta}$, we obtain that
\begin{eqnarray*}
&&-\int_{\partial'O_{\delta}}t^{1-2s}\partial_tU(X\cdot\nabla U)+\int_{\mathbf S_r^+\cap \{t>\delta\}}t^{1-2s}r\left|\frac{\partial U}{\partial\nu}\right|^2-\int_{\mathbf S_{\rho}^+\cap \{t>\delta\}}t^{1-2s}\rho\left|\frac{\partial U}{\partial\nu}\right|^2\notag\\
&&=\int_{O_{\delta}}t^{1-2s}\nabla U\cdot\nabla (X\cdot \nabla U)\notag\\
&&=\int_{O_{\delta}}t^{1-2s}|\nabla U|^2+\frac{1}{2}\int_{O_{\delta}}t^{1-2s}X\cdot\nabla(|\nabla U|^2)\notag\\
&&=\int_{O_{\delta}}t^{1-2s}|\nabla U|^2+\frac{1}{2}\int_{O_{\delta}}{\rm div}(t^{1-2s}|\nabla U|^2 X)-\frac{1}{2}\int_{O_{\delta}}{\rm div}(t^{1-2s} X)|\nabla U|^2\notag\\
&&=-\frac{n-2s}{2}\int_{O_{\delta}}t^{1-2s}|\nabla U|^2-\frac{1}{2}\int_{\partial'O_{\delta}}t^{2-2s}|\nabla U|^2
+\frac{1}{2}\int_{\partial''O_{\delta}}t^{1-2s}|\nabla U|^2X\cdot\nu.
\end{eqnarray*}
That is,
\begin{eqnarray}\label{e:pohozaev-2}
&&-\int_{\partial'O_{\delta}}\delta^{1-2s}\partial_tU(X\cdot\nabla U)+\int_{\mathbf S_r^+\cap \{t>\delta\}}t^{1-2s}r\left|\frac{\partial U}{\partial\nu}\right|^2-\int_{\mathbf S_{\rho}^+\cap \{t>\delta\}}t^{1-2s}\rho\left|\frac{\partial U}{\partial\nu}\right|^2\notag\\
&&=-\frac{n-2s}{2}\int_{O_{\delta}}t^{1-2s}|\nabla U|^2-\frac{1}{2}\int_{\partial'O_{\delta}}\delta^{2-2s}|\nabla U|^2\notag\\
&&\quad+\frac{1}{2}\int_{\mathbf S_r^+\cap \{t>\delta\}}t^{1-2s}r|\nabla U|^2-\frac{1}{2}\int_{\mathbf S_{\rho}^+\cap \{t>\delta\}}t^{1-2s}\rho|\nabla U|^2.
\end{eqnarray}
Rewriting the first term of the left side of (\ref{e:pohozaev-2}), we have
\begin{equation*}
\int_{\partial'O_{\delta}}\delta^{1-2s}\partial_tU(X\cdot\nabla U)=\delta^{2-2s}\int_{\partial'O_{\delta}}|\partial_t U|
+\int_{\partial'O_{\delta}}\delta^{1-2s}(x\cdot\nabla_x U)\partial_t U.
\end{equation*}
Thus
\begin{eqnarray*}
\frac{n-2s}{2}\int_{O_{\delta}}t^{1-2s}|\nabla U|^2&=&-\frac{\delta^{2-2s}}{2}\int_{\partial'O_{\delta}}|\nabla U|^2
+\delta^{2-2s}\int_{\partial'O_{\delta}}|\partial_tU|^2\notag\\
&&+\frac{r}{2}\int_{\mathbf S_r^+\cap \{t>\delta\}}t^{1-2s}|\nabla U|^2-r\int_{\mathbf S_r^+\cap \{t>\delta\}}t^{1-2s}\left|\frac{\partial U}{\partial\nu}\right|^2\notag\\
&&-\frac{\rho}{2}\int_{\mathbf S_{\rho}^+\cap \{t>\delta\}}t^{1-2s}|\nabla U|^2+\rho\int_{\mathbf S_{\rho}^+\cap \{t>\delta\}}t^{1-2s}\left|\frac{\partial U}{\partial\nu}\right|^2\notag\\
&&+\int_{\partial' O_{\delta}}\delta^{1-2s}(x\cdot\nabla_x U)\partial_t U.
\end{eqnarray*}

Since $U\in H^1(\mathbf R^{n+1},t^{1-2s})$, there exists a sequence $\delta_n\to 0$ such that
\begin{equation*}
\lim_{n\to\infty}\left(\frac{\delta_n^{2-2s}}{2}\int_{\partial'O_{\delta_n}}|\nabla U|^2+\delta_n^{2-2s}\int_{\partial'O_{\delta_n}}|\partial_tU|^2\right)=0.
\end{equation*}
From the Dominated Convergence Theorem and Lemma 3.3, Remark 3.6 in \cite{FallFelliCPDE2014}, it holds that
\begin{equation*}
\lim_{\delta\to 0}\int_{\partial' O_{\delta}}\delta^{1-2s}(x\cdot\nabla_x U)\partial_t U=-\kappa_s\int_{B_r\setminus B_{\rho}}(x\cdot\nabla_x U)\left(\dfrac{\lambda}{|x|^{\alpha}}U(x,0)+(|U|^{p-1}U)(x,0)\right).
\end{equation*}
Note that
\begin{eqnarray}\label{e:lim}
&&\int_{B_r\setminus B_{\rho}}(x\cdot\nabla_x U)\left(\dfrac{\lambda}{|x|^{\alpha}}U(x,0)+(|U|^{p-1}U)(x,0)\right)\\
&&=\frac{1}{2}\int_{B_r\setminus B_{\rho}}\frac{\lambda x}{|x|^{\alpha}}\cdot\nabla_x (U^2)+\frac{1}{p+1}\int_{B_r\setminus B_{\rho}}x\cdot\nabla_x(|U|^{p+1})\notag\\
&&:=T_1+T_2.
\end{eqnarray}
Then integrating by parts, we have
\begin{eqnarray*}
T_1&=&\frac{1}{2}\int_{B_r\setminus B_{\rho}}{\rm div}\left[\frac{\lambda x}{|x|^{\alpha}}U^2\right]-U^2{\rm div}\left[\frac{\lambda x}{|x|^{\alpha}}\right]\notag\\
&=&\frac{\lambda r}{2}\int_{\partial B_r}\frac{U^2}{|x|^{\alpha}}-\frac{\lambda \rho}{2}\int_{\partial B_{\rho}}\frac{U^2}{|x|^{\alpha}}
-\frac{n-\alpha}{2}\int_{B_r\setminus B_{\rho}}\frac{U^2}{|x|^{\alpha}}
\end{eqnarray*}
and
\begin{eqnarray*}
T_2&=&\frac{1}{p+1}\int_{B_r\setminus B_{\rho}}{\rm div}(|U|^{p+1}x)-|U|^{p+1}{\rm div}(x)\notag\\
&=&\frac{r}{p+1}\int_{\partial B_r}|U|^{p+1}-\frac{\rho}{p+1}\int_{\partial B_{\rho}}|U|^{p+1}-\frac{n}{p+1}\int_{B_r\setminus B_{\rho}}|U|^{p+1}.
\end{eqnarray*}
Therefore,
\begin{eqnarray}\label{e;pohozaev-3}
&&\frac{n-2s}{2}\int_{\mathcal B_r^+\setminus \mathcal B_{\rho}^+}t^{1-2s}|\nabla U|^2\\
&&=\frac{r}{2}\int_{\mathbf S_r^+}t^{1-2s}|\nabla U|^2-r\int_{\mathbf S_r^+}t^{1-2s}\left|\frac{\partial U}{\partial\nu}\right|^2
-\frac{\rho}{2}\int_{\mathbf S_{\rho}^+}t^{1-2s}|\nabla U|^2+\rho\int_{\mathbf S_{\rho}^+}t^{1-2s}\left|\frac{\partial U}{\partial\nu}\right|^2\notag\\
&&-\frac{\kappa_s\lambda r}{2}\int_{\partial B_r}\frac{U^2}{|x|^{\alpha}}+\frac{\kappa_s\lambda \rho}{2}\int_{\partial B_{\rho}}\frac{U^2}{|x|^{\alpha}}
+\kappa_s\frac{n-\alpha}{2}\int_{B_r\setminus B_{\rho}}\frac{U^2}{|x|^{\alpha}}\notag\\
&&-\frac{\kappa_sr}{p+1}\int_{\partial B_r}|U|^{p+1}+\frac{\kappa_s\rho}{p+1}\int_{\partial B_{\rho}}|U|^{p+1}+\frac{\kappa_sn}{p+1}\int_{B_r\setminus B_{\rho}}|U|^{p+1}.\notag
\end{eqnarray}
Since $U\in H^{1}(\mathbf R^{n+1}_+)$, $U(\cdot, 0)\in L^2(\mathbf R^{n},\frac{1}{|x|^{\alpha}})\cap L^{p+1}(\mathbf R^n)$, there exists a sequence $\rho_i\to 0$ such that when $i\to \infty$
\begin{equation*}
\frac{\rho_i}{2}\int_{\mathbf S_{\rho_i}^+}t^{1-2s}|\nabla U|^2+\rho_i\int_{\mathbf S_{\rho_i}^+}t^{1-2s}\left|\frac{\partial U}{\partial\nu}\right|^2\to 0,
\end{equation*}
\begin{equation*}
\frac{\kappa_s\lambda \rho_i}{2}\int_{\partial B_{\rho_i}}\frac{U^2}{|x|^{\alpha}}\to 0
\end{equation*}
and
\begin{equation*}
\frac{\kappa_s\rho_i}{p+1}\int_{\partial B_{\rho_i}}|U|^{p+1}\to 0.
\end{equation*}
Choosing $\rho=\rho_i$ and letting $i\to\infty$ in (\ref{e;pohozaev-3}), we have
\begin{eqnarray}\label{e:pohozaev}
\left(\frac{n-2s}{2}\right)\int_{\mathcal B_r^+}t^{1-2s}|\nabla U|^2&=&\frac{r}{2}\int_{\mathbf S_r^+}t^{1-2s}|\nabla U|^2-r\int_{\mathbf S_r^+}t^{1-2s}\left|\frac{\partial U}{\partial\nu}\right|^2\\
&-&\frac{\kappa_s\lambda r}{2}\int_{\partial B_r}\frac{U^2}{|x|^{\alpha}}
+\kappa_s\lambda\left(\frac{n-\alpha}{2}\right)\int_{B_r}\frac{U^2}{|x|^{\alpha}}\notag\\
&-&\frac{\kappa_sr}{p+1}\int_{\partial B_r}|U|^{p+1}+\frac{\kappa_sn}{p+1}\int_{B_r}|U|^{p+1}.\notag
\end{eqnarray}

We now prove that for a.e. $r\in (0,R)$,
\begin{equation}\label{e:pohozaev-1}
\int_{\mathcal B^+_r}t^{1-2s}|\nabla U|^2
=\int_{\mathbf S_r^+}t^{1-2s}\frac{\partial U}{\partial\nu}U+\kappa_s\lambda\int_{B_r}\frac{U^2}{|x|^{\alpha}}+\kappa_s\int_{B_r}|U|^{p+1}.
\end{equation}
Indeed, let $\eta_k(\rho)$ be a sequence of cut-off functions such that $\eta_k(\rho)=1$ if $0\le\rho<r-\frac{1}{k}$, $\eta_k(\rho)=0$ if $\rho>r$ and
$\eta_k=k(r-\rho)$ if $r-\frac{1}{k}\le\rho\le r$. Choosing $\eta_k(|X|)U(X)$ as a testing function in (\ref{e:extention-bounded-domain}), we have that
\begin{eqnarray}\label{e:wbrae1}
&&\int_{\mathcal B_r^+ }t^{1-2s}\nabla U(X)\cdot \nabla \left[\eta_k(|X|)U(X)\right]\notag\\
&&=\kappa_s\int_{B_r}\frac{\lambda}{|x|^{\alpha}}U^2(x,0)\eta_k(|x|)+\kappa_s\int_{B_r}|U|^{p+1}(x,0)\eta_k(|x|).
\end{eqnarray}
A direct calculation yields that
\begin{eqnarray*}
&&\int_{\mathcal B_r^+ }t^{1-2s}\nabla U(X)\cdot \nabla \left[\eta_k(|X|)U(X)\right]\\
&&=\int_{\mathcal B_r^+ }t^{1-2s}|\nabla U(X)|^2\eta_k(|X|)-k\int_{\mathcal B_r^+\setminus\mathcal B_{r-\frac{1}{k}}}t^{1-2s}\frac{\partial U}{\partial \nu}U\notag\\
&&=\int_{\mathcal B_r^+ }t^{1-2s}|\nabla U(X)|^2\eta_k(|X|)-k\int_{r-\frac{1}{k}}^r\left(\int_{\mathbf S_{\rho}^+}t^{1-2s}\frac{\partial U}{\partial \nu}UdS\right)d\rho.\notag
\end{eqnarray*}
Note that $t^{1-2s}\frac{\partial U}{\partial \nu}U\in L^1(\mathbf R^{n+1}_+)$, so for a.e. $r\in (0,R)$ there holds
\begin{equation*}
k\int_{r-\frac{1}{k}}^r\left(\int_{\mathbf S_{\rho}^+}t^{1-2s}\frac{\partial U}{\partial \nu}UdS\right)d\rho\to \int_{\mathbf S_{r}^+}t^{1-2s}\frac{\partial U}{\partial \nu}UdS,\quad \mbox{ as }k\to\infty.
\end{equation*}
On the other hand, as $k\to \infty$, we have
\begin{multline*}
\kappa_s\int_{B_r}\frac{\lambda}{|x|^{\alpha}}U^2(x,0)\eta_k(|x|)+\kappa_s\int_{B_r}|U|^{p+1}(x,0)\eta_k(|x|)\\
\to \kappa_s\int_{B_r}\frac{\lambda}{|x|^{\alpha}}|U|^2(x,0)+\kappa_s\int_{B_r}|U|^{p+1}(x,0)
\end{multline*}
and
\begin{equation}\label{e:wbrae2}
\int_{\mathcal B_r^+ }t^{1-2s}|\nabla U(X)|^2\eta_k(|X|)\to \int_{\mathcal B_r^+ }t^{1-2s}|\nabla U(X)|^2.
\end{equation}
Therefore, from (\ref{e:wbrae1})-(\ref{e:wbrae2}), we obtain (\ref{e:pohozaev-1}).

Finally, by (\ref{e:pohozaev}) and (\ref{e:pohozaev-1}), we have (\ref{e:pohozaev0}). This completes the proof.
\end{proof}

\begin{proof}[Proof of Corollary \ref{c:nonexistence-g}]
(1) From the Pohozaev-type identity (\ref{e:pohozaev0}), we deduce that if $U$ is a nontrivial solution to (\ref{e:mail-extention}), then for a.e. $r>0$ there holds
\begin{eqnarray}\label{e:pU1}
&&\kappa_s\lambda\left(\frac{2s-\alpha}{2}\right)\int_{B_r}\frac{U^2}{|x|^{\alpha}}+\kappa_s\left(\frac{n}{p+1}-\frac{n-2s}{2}\right)\int_{B_r}|U|^{p+1}\\
&&=\frac{r}{2}\int_{\mathbf S_r^+}t^{1-2s}|\nabla U|^2-r\int_{\mathbf S_r^+}t^{1-2s}\left|\frac{\partial U}{\partial\nu}\right|^2
-\frac{\kappa_s\lambda r}{2}\int_{\partial B_r}\frac{U^2}{|x|^{\alpha}}\notag\\
&&-\frac{\kappa_sr}{p+1}\int_{\partial B_r}|U|^{p+1}-\frac{n-2s}{2}\int_{\mathbf S^+_r}t^{1-2s}\frac{\partial U}{\partial\nu}U.\notag
\end{eqnarray}
By assumption of $\lambda,\,\alpha,\,p$, the left side of (\ref{e:pU1}) is positive.

Now we claim that there exists a sequence $r_i\to \infty$ such that
\begin{equation*}
\frac{r_i}{2}\int_{\mathbf S_{r_i}^+}t^{1-2s}|\nabla U|^2\to 0,\quad \mbox{as }i\to \infty.
\end{equation*}
Indeed, if this is not true, then there is a constant $C>0$ such that
\begin{equation*}
\frac{r}{2}\int_{\mathbf S_r^+}t^{1-2s}|\nabla U|^2\ge C>0,\quad \mbox{as }r\to \infty.
\end{equation*}
Thus, there exists an $r_1$ sufficiently large such that for all $r>r_1$,
\begin{equation}\label{e:rate-r}
\int_{\mathbf S_r^+}t^{1-2s}|\nabla U|^2\ge \frac{C}{r}.
\end{equation}
Since $U\in H^1(\mathbf R^{n+1},t^{1-2s})$, it holds that
\begin{equation*}
\int_{0}^{\infty}dr\int_{\mathbf S_r^+}t^{1-2s}|\nabla U|^2<\infty.
\end{equation*}
This is impossible when (\ref{e:rate-r}) holds. So we get our claim. Similarly, because $U\in H^1(\mathbf R^{n+1},t^{1-2s})$ and $U(\cdot, 0)\in H^s(\mathbf R^n)\cap L^2(\mathbf R^n,\frac{1}{|x|^{\alpha}})$, there exists a sequence $r_i\to \infty$ such that
\begin{eqnarray*}
r_i\int_{\mathbf S_{r_i}^+}t^{1-2s}\left|\frac{\partial U}{\partial\nu}\right|^2
+\frac{\kappa_s\lambda r_i}{2}\int_{\partial B_{r_i}}\frac{U^2}{|x|^{\alpha}}
+\frac{\kappa_sr_i}{p+1}\int_{\partial B_{r_i}}U^{p+1}+\frac{n-2s}{2}\int_{\mathbf S^+_{r_i}}t^{1-2s}\frac{\partial U}{\partial\nu}U\to 0.
\end{eqnarray*}
Therefore, choosing $r=r_i$ and letting $i\to \infty$ in (\ref{e:pU1}), we have
\begin{equation*}
\int_{\mathbf R^n}|U|^{p+1}=0.
\end{equation*}
That means that $U(\cdot,0)\equiv 0$ in $\mathbf R^n$. From (\ref{e:pohozaev-1}), we have $U\equiv 0$ in $\mathbf R^{n+1}_+$. This yields the result of (1).

(2) Assume that $U\in H^1(\mathbf R^{n+1}_+,t^{1-2s})\cap L^{p+1}(\mathbf R^n)$. When $\alpha=2s$, by (\ref{e:pohozaev0}), we have for a.e. $r>0$
\begin{eqnarray}\label{e:pU2}
&&\kappa_s\left(\frac{n}{p+1}-\frac{n-2s}{2}\right)\int_{B_r}|U|^{p+1}\\
&&=\frac{r}{2}\int_{\mathbf S_r^+}t^{1-2s}|\nabla U|^2-r\int_{\mathbf S_r^+}t^{1-2s}\left|\frac{\partial U}{\partial\nu}\right|^2
-\frac{\kappa_s\lambda r}{2}\int_{\partial B_r}\frac{U^2}{|x|^{2s}}\notag\\
&&-\frac{\kappa_sr}{p+1}\int_{\partial B_r}|U|^{p+1}-\frac{n-2s}{2}\int_{\mathbf S^+_r}t^{1-2s}\frac{\partial U}{\partial\nu}U.\notag
\end{eqnarray}
A similar argument as in the proof of (1) yields that $U\equiv 0$ in $\mathbf R^{n+1}_+$.

(3) Let $U\in H^s(\mathbf R^n)\cap L^{2}(\mathbf R^n,\frac{1}{|x|^{\alpha}})$. From Pohozaev type identity (\ref{e:pohozaev0}), for $p= 2^*(s)-1$, it holds that
\begin{eqnarray}\label{e:pU3}
\kappa_s\lambda\left(\frac{2s-\alpha}{2}\right)\int_{B_r}\frac{U^2}{|x|^{\alpha}}
&=&\frac{r}{2}\int_{\mathbf S_r^+}t^{1-2s}|\nabla U|^2-r\int_{\mathbf S_r^+}t^{1-2s}\left|\frac{\partial U}{\partial\nu}\right|^2
-\frac{\kappa_s\lambda r}{2}\int_{\partial B_r}\frac{U^2}{|x|^{\alpha}}\notag\\
&&-\frac{\kappa_sr}{p+1}\int_{\partial B_r}|U|^{p+1}-\frac{n-2s}{2}\int_{\mathbf S^+_r}t^{1-2s}\frac{\partial U}{\partial\nu}U.\notag
\end{eqnarray}
Similarly, we have $U\equiv 0$ in $\mathbf R^{n+1}_+$ by an argument in the proof of (1). This completes the proof.
\end{proof}


\section{Regularity of the solutions}\label{s:regularity}

In this section, we will give the proof of the regularity conclusion in Theorem \ref{t:regulatiry}. Through out this section and Section \ref{s:symmetry} below,
we omit the constant $\kappa_s$ for simplicity since it is not essential in the proof.

Let $D\subset \mathbf R^{n+1}_+$ be a bounded domain with $\partial' D\ne \emptyset$ and $0\notin \partial'D$, $a\in L^{\frac{2n}{n+2s}}(\partial'D)$ and $b\in L^1_{loc}(\partial' D)$. Consider the following boundary value problem,
\begin{equation}\label{e:extention_bounded}
\left\{\begin{array}{rl}
         {\rm div}(t^{1-2s}\nabla U(x,t))=0,& {\rm in}\, D, \\
         \displaystyle-\lim_{t\to 0^+}t^{1-2s}\partial_t U(x,t)=a(x)U(x,0)+b(x) ,&{\rm on}\,\partial' D.
       \end{array}
\right.
\end{equation}
\begin{dfn}\label{d:weak-solution}
$U\in H^1(D,t^{1-2s})$ is a weak solution (resp. supersolution, subsolution) of (\ref{e:extention_bounded}), if for every nonnegative $V\in C_c^{\infty}(D\cup\partial'D)$, one has
\begin{equation*}
\int_{D}t^{1-2s}\nabla U \nabla V=(\mbox{ resp. }\ge,\,\le)\int_{\partial'D}aUV+bV.
\end{equation*}
\end{dfn}

Let $Q_R=B_R\times (0,R)$, $f^+=\max(f,0)$ and $f^-=\max(-f,0)$. First, we recall
\begin{lemma}\label{l:moser-greater}{\rm (\cite[Lemma 2.8]{JinLiXiongJEMS2014})}
Suppose that $a\in L^{\frac{n}{2s}}(B_1)$, $b\in L^p(B_1)$ with $p>\frac{n}{2s}$ and $U\in H^1(Q_1,t^{1-2s})$ is a weak solution of (\ref{e:extention_bounded}) in $Q_1$. There exists $\delta>0$ which depends only on $n$ and $s$ such that if $\|a^+\|_{L^{\frac{n}{2s}}(B_1)}<\delta$, then
\begin{equation*}
\|U^+(\,\cdot\,,0)\|_{L^q(\partial'Q_{1/2})}\le C\left(\|U^+\|_{H^1(Q_1,t^{1-2s})}+\|b^+\|_{L^p(B_1)}\right),
\end{equation*}
where $C>0$ only depends on $n,\,p,\,s\,,\delta$ and $q=\min\left(\frac{2(n+1)}{n-2s},\frac{n(p-1)}{(n-2s)p}\cdot\frac{2n}{n-2s}\right)$.
\end{lemma}

\begin{lemma}\label{p:harnack_1}{\rm (\cite[Proposition 2.6]{JinLiXiongJEMS2014})}
Suppose that $a,\,b\in L^p(B_1)$ for some $p>\frac{n}{2s}$. Let $U\in H^1(Q_1,t^{1-2s})$ be a nonnegative weak solution to (\ref{e:extention_bounded}) in $Q_1$. Then

(1) $\forall \nu>0$, there holds
\begin{equation*}
\sup_{Q_{1/2}}U\le C(\|U\|_{L^{\nu}(Q_1,t^{1-2s})}+\|b^+\|_{L^p(B_1)}),
\end{equation*}
where $C>0$ only depends on $n,\,s,\,p,\,\nu$ and $\|a^+\|_{L^p(B_1)}$.

(2)
We have the following Harnack inequality,
\begin{equation*}
\sup_{Q_{1/2}}U\le C(\inf_{Q_{1/2}}U+\|b\|_{L^p(B_1)}),
\end{equation*}
where $C>0$ only depends on $n,\,s,\,p$ and $\|a\|_{L^p(B_1)}$. Moreover, any weak solution $U$ of (\ref{e:extention_bounded}) is in $C^\varrho(\overline{Q_{1/2}})$ for some $\varrho\in (0, 1)$ only depending on $n$, $s$, $p$ and $\|a\|_{L^p(B_1)}$.
\end{lemma}

Let $D$ be a domain in $\mathbf R^{n+1}_+$ with $\partial' D\ne \emptyset$ and $0\notin \partial' D$. We consider the following problem,
\begin{equation}\label{e:extention_v_b}
\left\{\begin{array}{rl}
         {\rm div}(t^{1-2s}\nabla U(x,t))=0,& {\rm in}\, D, \\
         \displaystyle-\lim_{t\to 0^+}t^{1-2s}\partial_t U(x,t)=\frac{\lambda}{|x|^{2s}}U(x,0)+U^{\frac{n+2s}{n-2s}}(x,0) ,&{\rm on}\,\partial' D.
       \end{array}
\right.
\end{equation}

\begin{prop}\label{p:Harnack}
Let $D=(B_R\setminus B_r)\times (0,1)$ where $0<r<R$, $0<\delta<\frac{R-r}{2}$ and $K_{\delta/2}=(B_{R-\delta/2}\setminus B_{r+\delta/2})\times (0,1/2)$.  Suppose $U\in H^1(t^{1-2s},D)$ is a nonnegative weak solution of (\ref{e:extention_v_b}). Then we have
\begin{equation*}
\sup_{K_{\delta/2}}U\le C\inf_{K_{\delta/2}} U,
\end{equation*}
where $C$ depends only on $n,\,s$ and $\delta$. Moreover, any weak solution $U$ of (\ref{e:extention_v_b}) is in $C^\varrho(\overline{K_{\delta/2}})$ for some $\varrho\in (0, 1)$ depending only on $n$, $s$ and $\delta$.
\end{prop}
\begin{proof}
Note that $\partial'D=B_R\setminus B_r$, it follows that $\frac{1}{|x|^{n-2s}}\in L^{\infty}(\partial'D)$. Since $U(\,\cdot\,,0)\in H^s(\partial'D)\subset L^{\frac{2n}{n-2s}}$ (\cite[Proposition 2.1]{JinLiXiongJEMS2014}), $U^{\frac{4s}{n-2s}}(\,\cdot\,,0)\in L^{\frac{n}{2s}}(\partial'D)$. Hence from Lemma \ref{l:moser-greater} (in our case $b=0$, so $p$ can be any positive number $> \frac{n}{2s}$), we have
$$U(\,\cdot\,,0)\in L^q(\partial'K_{3\delta/4}),$$
where $q=\min\left(\frac{2(n+1)}{n-2s},\frac{n(p-1)}{(n-2s)p}\cdot\frac{2n}{n-2s}\right)>\frac{2n}{n-2s}$. Therefore
\begin{equation*}
U^{\frac{4s}{n-2s}}(\,\cdot\,,0)\in L^{q'}(\partial'K_{3\delta/4})
\end{equation*}
with $q'>\frac{n}{2s}$. Finally, by Lemma \ref{p:harnack_1}, we obtain the desired estimate.
\end{proof}

\begin{proof}[Proof of Theorem \ref{t:regulatiry} (1)]
As in the proof of Proposition \ref{p:Harnack}, set $a(x) = \frac{\lambda}{|x|^{n-2s}} + U^{\frac{4s}{n-2s}}(x,0)$ and $b(x) = 0$. Then for any $x\in \mathbf R^n\setminus\{0\}$, by Proposition \ref{p:Harnack}, $a(x)\in C^\varrho(\overline{B_r(x)})$ for some small positive number $r$ and $\varrho\in (0, 1)$ which depends only on $n$, $s$, $\delta$, $p$. From Theorem 2.11 of \cite{JinLiXiongJEMS2014}, $U\in C^{\varrho + 2s}(\overline{\mathcal B_{r/2}^+(x)})$. By a bootstrapping argument, we then have $U\in C^{\infty}(\overline{\mathbf R^{n+1}_+}\setminus\{0\})$. This completes the proof.
\end{proof}


\section{Symmetry of the solutions: the method of moving spheres}\label{s:symmetry}

In this section, we prove Theorem \ref{t:regulatiry} (2) by the method of moving spheres. 

\subsection{Technique lemmas}
We begin with some technique lemmas.

Let $X\in \partial\mathbf R^{n+1}_+$, $\rho>0$ and $U_{X,\rho}$ be the Kelvin transformation defined by
\begin{equation*}
U_{X,\rho}(\xi)=\left(\frac{\rho}{|\xi-X|}\right)^{n-2s}U\left(X+\frac{\rho^2(\xi-X)}{|\xi-X|^2}\right),\quad \xi\in \overline{\mathbf R^{n+1}_+}\setminus \{X\}.
\end{equation*}

\begin{lemma}
If $U$ satisfies (\ref{e:mail-extention-c}), then $U_{X,\rho}$ is a solution of the following equation
\begin{equation}\label{e:kte}
\left\{\begin{array}{rl}
         {\rm div}(t^{1-2s}\nabla U_{X_0,\rho})(\xi)=0,\,{\rm in}\,\mathbf R^{n+1}_+, \\
         \displaystyle-\lim_{t\to 0^+}t^{1-2s}\partial_t U_{X_0,\rho}(\xi)=\left(\frac{\rho}{|x-x_0|}\right)^{4s}\frac{\lambda}{|x_{\rho,x_0}|^{2s}}U_{X_0,\rho}(x_{\rho,x_0},0)
+U_{X_0,\rho}^{\frac{n+2s}{n-2s}}(x_{\rho,x_0},0),\,{\rm on}\,\mathbf R^n,
       \end{array}
\right.
\end{equation}
where $\xi=(x,t)$, $X_0=(x_0,0)$ and
\begin{equation}\label{e:xxl}
x_{\rho,x_0}:=x_0+\frac{\rho^2(x-x_0)}{|x-x_0|^2}.
\end{equation}
\end{lemma}
The result of this lemma is the case of $\alpha=2s$, $p=2^*(s)$ in Lemma \ref{l:akte}. For the detailed proof, see Appendix \ref{s:kt}.
\begin{lemma}\label{l:kxx0l}
Let $0<\rho<|x_0|$. Then one has
\begin{equation}\label{e:kxgl}
\left(\frac{\rho}{|x-x_0|}\right)^{4s}\left(\frac{1}{|x_{\rho,x_0}|^{2s}}\right)\le \frac{1}{|x|^{2s}},\quad \forall\, \rho<|x-x_0|<|x_0|
\end{equation}
and
\begin{equation}\label{e:kxll}
\left(\frac{\rho}{|x-x_0|}\right)^{4s}\left(\frac{1}{|x_{\rho,x_0}|^{2s}}\right)\ge \frac{1}{|x|^{2s}},\quad \forall\, 0<|x-x_0|<\rho.
\end{equation}
\end{lemma}
\begin{proof}
A direct computation yields that
\begin{eqnarray*}
&&\left|x+\left(\frac{|x-x_0|^2}{\rho^2}-1\right)x_0\right|^2-|x|^2\\
&&=\left(\frac{|x-x_0|^2}{\rho^2}-1\right)\left(2\langle x,x_0\rangle
-\langle x_0,x_0\rangle+\left(\langle x,x\rangle-2\langle x,x_0\rangle +\langle x_0,x_0\rangle\right)\frac{\langle x_0,x_0\rangle}{\rho^2}\right)\notag.
\end{eqnarray*}
Since $0<\rho<|x_0|$, it holds that
\begin{eqnarray*}
&&2\langle x,x_0\rangle
-\langle x_0,x_0\rangle+\left(\langle x,x\rangle-2\langle x,x_0\rangle +\langle x_0,x_0\rangle\right)\frac{\langle x_0,x_0\rangle}{\rho^2}\notag\\
&&\ge 2\langle x,x_0\rangle
-\langle x_0,x_0\rangle+\left(\langle x,x\rangle-2\langle x,x_0\rangle +\langle x_0,x_0\rangle\right)\notag\\
&&=\langle x,x\rangle> 0.
\end{eqnarray*}
Hence, for all $\rho<|x-x_0|<|x_0|$,
\begin{eqnarray*}
\left|x+\left(\frac{|x-x_0|^2}{\rho^2}-1\right)x_0\right|^2-|x|^2\ge 0
\end{eqnarray*}
and,
for all $0<|x-x_0|<\rho$,
\begin{eqnarray*}
\left|x+\left(\frac{|x-x_0|^2}{\rho^2}-1\right)x_0\right|^2-|x|^2\le 0.
\end{eqnarray*}
Then we have (\ref{e:kxgl}) and (\ref{e:kxll}) from (\ref{e:xxl}). This completes the proof.
\end{proof}

Next, we need two versions of maximum principle.
\begin{lemma}\label{l:non-negative}{\rm (\cite[Lemma 2.5]{JinLiXiongJEMS2014})}
Let $U\in H^1(D,t^{1-2s})$ be a weak super-solution of (\ref{e:extention_bounded}) in $D$ with $a\equiv b\equiv0$. If $U\ge 0$ on $\partial'' D$ in the trace sense, then $U\ge 0$ in $D$.
\end{lemma}

\begin{lemma}\label{p:pn0}\cite[Proposition 3.1]{JinLiXiongJEMS2014}
Assume that $U\in H^1(\mathcal B^+_1\setminus\overline{\mathcal B}^+_{\varepsilon},|t|^{1-2s})$ is a solution of the following boundary value problem
\begin{equation*}
\left\{\begin{array}{ll}
         {\rm div}(t^{1-2s}\nabla U)\le 0 & \mbox{ in }\mathcal B^+_1, \\
         -\displaystyle\lim_{t\to 0}t^{1-2s}\partial_t U(x,t)\ge 0 & \mbox{ on } B_1\setminus\bar{B}_{\varepsilon},
       \end{array}
\right.
\end{equation*}
for every $\varepsilon\in (0,1)$. If $U\in C(\mathcal B^+_1\cup B_1\setminus\{0\})$ and $U>0$ in $\mathcal B^+_1\cup B_1\setminus\{0\}$, then
\begin{equation*}
\displaystyle \liminf_{|X|\to 0} U(X)>0.
\end{equation*}
\end{lemma}

\subsection{The method of moving spheres}
We verify the symmetry property of the solutions for (\ref{e:mail-extention-c}).
Firstly, we prove
\begin{prop}\label{p:beginning}
For each $x_0\ne 0$, there exists a constant $\rho(x_0)>0$ depending on $x_0$ such that for any $0<\rho<\rho(x_0)$, it holds that
\begin{equation*}
U_{X_0,\rho}(\xi)\le U(\xi)\mbox{ in } \mathbf R^{n+1}_+\setminus (\overline{\mathcal B}^+_{\rho}(X_0)\cup\{0\}),
\end{equation*}
where $X_0=(x_0,0)$.
\end{prop}
We divide the proof by several lemmas.
\begin{lemma}\label{l:lemma5.2}
For each $\rho_2\in (0,|X_0|)$, there exists $\rho_1 > 0$ small enough such that, when $\rho\in (0,\rho_1)$,
\begin{equation*}
U_{X_0,\rho}\le U\mbox{ on }\partial''\left[\mathcal B_{\rho_2}^+(X_0)\setminus\mathcal B_{\rho}^+(X_0)\right].
\end{equation*}
\end{lemma}
\begin{proof}
Assume that $\xi\in \partial''\mathcal B_{\rho_2}^+(X_0)$. Then for every $\rho$ satisfying $0<\rho<\rho_1<\rho_2<|X_0|$, we have
\begin{equation*}
X_0+\frac{\rho^2(\xi-X_0)}{|\xi-X_0|}\in \mathcal B_{\rho_2}^+(X_0).
\end{equation*}
Therefore, by Proposition \ref{p:Harnack}, we can choose $\rho_1$ small enough such that
\begin{eqnarray*}
U_{X_0,\rho}(\xi)&=&\left(\frac{\rho}{|\xi-X_0|}\right)^{n-2s}U\left(X_0+\frac{\rho^2(\xi-X_0)}{|\xi-X_0|}\right)\notag\\
&\le&\left(\frac{\rho_1}{\rho_2}\right)^{n-2s}\sup_{\overline{\mathcal B}_{\rho_2}^+(X_0)}U\le \inf_{\partial''\mathcal B_{\rho_2}^+(X_0) }U\notag\\
&\le & U(\xi).
\end{eqnarray*}
Note that on $\partial''\mathcal B_{\rho}^+(X_0)$, $U_{X_0,\rho}=U$ holds obviously. This completes the proof of this lemma.
\end{proof}

\begin{lemma}\label{l:l2small}
If $\rho_2\in (0,|X_0|)$ is small enough, then there exists a constant $\rho_1 = \rho_1(\rho_2)$ depending on $\rho_2$ and satisfying  $0<\rho_1(\rho_2)<\rho_2$ such that, for all $\rho \in (0,\rho_1(\rho_2))$ and $\xi\in \mathcal B^+_{\rho_2}(X_0)\setminus \mathcal B^+_{\rho}(X_0)$, it holds that
\begin{equation*}
U_{X_0,\rho}(\xi)\le U(\xi).
\end{equation*}
\end{lemma}
\begin{proof}
From (\ref{e:mail-extention-c}) and (\ref{e:kte}), we have
\begin{equation}\label{e:divulambda-u}
{\rm div}\left(t^{1-2s}\nabla (U_{X_0,\rho}-U)\right)=0\quad\mbox{in  }\mathcal B^+_{\rho_2}(X_0)\setminus \overline{\mathcal B}^+_{\rho}(X_0)
\end{equation}
and
\begin{eqnarray}\label{e:b-ulambda-u}
-\lim_{t\to 0^+}\partial_t(U_{X_0,\rho}-U)
&=&\left[\left(\frac{\rho}{|x-x_0|}\right)^{4s}\frac{a}{|x_{\rho,x_0}|^{2s}}U_{X_0,\rho}(x_{\rho,x_0},0)-\frac{a}{|x|^{2s}}U(x,0)\right]\\
&&+\left[U_{X_0,\rho}^{\frac{n+2s}{n-2s}}(x_{\rho,x_0},0)-U^{\frac{n+2s}{n-2s}}(x,0)\right]\mbox{  on  }\partial'\left[\mathcal B^+_{\rho_2}(X_0)\setminus \overline{\mathcal B}^+_{\rho}(X_0)\right].\notag
\end{eqnarray}
Set $(U_{X_0,\rho}-U)^+:=\max(0,U_{X_0,\rho}-U)$, by Lemma \ref{l:lemma5.2},  it equals to $0$ on $\partial''\left[\mathcal B^+_{\rho_2}(X_0)\setminus \overline{\mathcal B}^+_{\rho}(X_0)\right]$. Let $(U_{X_0,\rho}-U)^+$ be a test function in the definition of weak solution for (\ref{e:extention_v_b}),
then from (\ref{e:divulambda-u}), (\ref{e:b-ulambda-u}) and Definition \ref{d:weak-solution}, we have
\begin{eqnarray*}
&&-\int_{\mathcal B^+_{\rho_2}(X_0)\setminus\mathcal B^+_{\rho}(X_0)}t^{1-2s}|\nabla(U_{X_0,\rho}-U)^+|^2\\
&&=\int_{\partial'\left[\mathcal B^+_{\rho_2}(X_0)\setminus \overline{\mathcal B}^+_{\rho}(X_0)\right]}
\left[\left(\frac{\rho}{|x-x_0|}\right)^{4s}\frac{\lambda}{|x_{\rho,x_0}|^{2s}}U_{X_0,\rho}(x_{\rho,x_0},0)-\frac{\lambda}{|x|^{2s}}U(x,0)\right]
(U_{X_0,\rho}-U)^+\notag\\
&&\quad +\int_{\partial'\left[\mathcal B^+_{\rho_2}(X_0)\setminus \overline{\mathcal B}^+_{\rho}(X_0)\right]}\left[U_{X_0,\rho}^{\frac{n+2s}{n-2s}}(x_{\rho,x_0},0)-U^{\frac{n+2s}{n-2s}}(x,0)\right](U_{X_0,\rho}-U)^+\notag\\
&&:=T_1+T_2.\notag
\end{eqnarray*}
In the following, we estimate $T_1$ and $T_2$ respectively.

(1) $T_1$: A direct computation yields that
\begin{eqnarray*}
T_1&=&\int_{\partial'\left[\mathcal B^+_{\rho_2}(X_0)\setminus \overline{\mathcal B}^+_{\rho}(X_0)\right]}
\left[\left(\frac{\rho}{|x-x_0|}\right)^{4s}\frac{\lambda}{|x_{\rho,x_0}|^{2s}}\left(U_{X_0,\rho}(x_{\rho,x_0},0)-U(x,0)\right)\right](U_{X_0,\rho}-U)^+\notag\\
&&+\int_{\partial'\left[\mathcal B^+_{\rho_2}(X_0)\setminus \overline{\mathcal B}^+_{\rho}(X_0)\right]}\left[\left(\frac{\rho}{|x-x_0|}\right)^{4s}\frac{\lambda}{|x_{\rho,x_0}|^{2s}}-\frac{\lambda}{|x|^{2s}}\right]U(x,0)
(U_{X_0,\rho}-U)^+\notag\\
&=&\int_{\partial'\left[\mathcal B^+_{\rho_2}(X_0)\setminus \overline{\mathcal B}^+_{\rho}(X_0)\right]}\left(\frac{\rho}{|x-x_0|}\right)^{4s}\frac{\lambda}{|x_{\rho,x_0}|^{2s}}\left((U_{X_0,\rho}-U)^+\right)^2\notag\\
&&+\int_{\partial'\left[\mathcal B^+_{\rho_2}(X_0)\setminus \overline{\mathcal B}^+_{\rho}(X_0)\right]}\left[\left(\frac{\rho}{|x-x_0|}\right)^{4s}\frac{\lambda}{|x_{\rho,x_0}|^{2s}}-\frac{\lambda}{|x|^{2s}}\right]U(x,0)
(U_{X_0,\rho}-U)^+.
\end{eqnarray*}
Since $\rho_2\le |x_0|=|X_0|$, then,
by Lemma \ref{l:kxx0l}, we have
\begin{equation*}
\left(\frac{\rho}{|x-x_0|}\right)^{4s}\frac{\lambda}{|x_{\rho,x_0}|^{2s}}-\frac{\lambda}{|x|^{2s}}\le 0,\quad \mbox{in } \partial'\left[\mathcal B^+_{\rho_2}(X_0)\setminus \overline{\mathcal B}^+_{\rho}(X_0)\right].
\end{equation*}
Therefore,
\begin{eqnarray*}
T_1\le \int_{\partial'\left[\mathcal B^+_{\rho_2}(X_0)\setminus \overline{\mathcal B}^+_{\rho}(X_0)\right]}\left(\frac{\rho}{|x-x_0|}\right)^{4s}\frac{\lambda}{|x_{\rho,x_0}|^{2s}}\left((U_{X_0,\rho}-U)^+\right)^2.
\end{eqnarray*}
By H\"older inequality, one has
\begin{eqnarray*}
&&T_1\leq \\
&&\left(\int_{\partial'\left[\mathcal B^+_{\rho_2}(X_0)\setminus \overline{\mathcal B}^+_{\rho}(X_0)\right]}\frac{\rho^{2n}}{|x-x_0|^{2n}}\frac{\lambda^{\frac{n}{2s}}}{|x_{\rho,x_0}|^{n}}\right)^{\frac{2s}{n}}
\left(\int_{\partial'\left[\mathcal B^+_{\rho_2}(X_0)\setminus \overline{\mathcal B}^+_{\rho}(X_0)\right]}\left((U_{X_0,\rho}-U)^+\right)^{\frac{2n}{n-2s}}\right)^{\frac{n-2s}{n}}.
\end{eqnarray*}

(2) $T_2$:
By the mean value theorem and H\"older inequality, we have
\begin{eqnarray*}
T_2&\le&C\int_{\partial'\left[\mathcal B^+_{\rho_2}(X_0)\setminus \overline{\mathcal B}^+_{\rho}(X_0)\right]} U_{X_0,\rho}^{\frac{4s}{n-2s}}\left((U_{X_0,\rho}-U)^+\right)^2\notag\\
&\le& C\left(\int_{\partial'\left[\mathcal B^+_{\rho_2}(X_0)\setminus \overline{\mathcal B}^+_{\rho}(X_0)\right]} U_{X_0,\rho}^{\frac{2n}{n-2s}}\right)^{\frac{2s}{n}}\left(\int_{\partial'\left[\mathcal B^+_{\rho_2}(X_0)\setminus \overline{\mathcal B}^+_{\rho}(X_0)\right]}\left((U_{X_0,\rho}-U)^+\right)^{\frac{2n}{n-2s}}\right)^{\frac{n-2s}{n}}\notag\\
&\le& C \left(\int_{\mathcal{B}^+_{\rho_2}(X_0)} U^{\frac{2n}{n-2s}}\right)^{\frac{2s}{n}}\left(\int_{\partial'\left[\mathcal B^+_{\rho_2}(X_0)\setminus \overline{\mathcal B}^+_{\rho}(X_0)\right]}\left((U_{X_0,\rho}-U)^+\right)^{\frac{2n}{n-2s}}\right)^{\frac{n-2s}{n}},
\end{eqnarray*}
where $C > 0$ is a constant depending only on $n$ and $s$.

By Proposition \ref{p:trace},
\begin{eqnarray*}
T_1+T_2&\le&C\left[\left(\int_{\partial'\left[\mathcal B^+_{\rho_2}(X_0)\setminus \overline{\mathcal B}^+_{\rho}(X_0)\right]}\frac{\rho^{2n}}{|x-x_0|^{2n}}\frac{\lambda^{\frac{n}{2s}}}{|x_{\rho,x_0}|^{n}}\right)^{\frac{2s}{n}}+\left(\int_{\mathcal{B}^+_{\rho_2}(X_0)} U^{\frac{2n}{n-2s}}\right)^{\frac{2s}{n}}\right]\notag\\
&&\,\,\,\,\times\int_{\mathcal B^+_{\rho_2}(X_0)\setminus \overline{\mathcal B}^+_{\rho}(X_0)}t^{1-2s}|\nabla(U_{X_0,\rho}-U)^+|^2.
\end{eqnarray*}
Thus, if we choose $\rho_2$ small enough such that
\begin{equation*}
C\left[\left(\int_{\partial'\left[\mathcal B^+_{\rho_2}(X_0)\setminus \overline{\mathcal B}^+_{\rho}(X_0)\right]}\frac{\rho^{2n}}{|x-x_0|^{2n}}\frac{\lambda^{\frac{n}{2s}}}{|x_{\rho,x_0}|^{n}}\right)^{\frac{2s}{n}}+\left(\int_{\mathcal{B}^+_{\rho_2}(X_0)} U^{\frac{2n}{n-2s}}\right)^{\frac{2s}{n}}\right]<\frac{1}{2},
\end{equation*}
we get
\begin{equation*}
\nabla(U_{X_0,\rho}-U)^+=0 \mbox{ in } \mathcal B^+_{\rho_2}(X_0)\setminus\mathcal B^+_{\rho}(X_0).
\end{equation*}
Since $(U_{X_0,\rho}-U)^+=0$ on $\partial''(\mathcal B^+_{\rho_2}(X_0)\setminus\mathcal B^+_{\rho}(X_0))$, we have
\begin{equation*}
(U_{X_0,\rho}-U)^+=0 \mbox{ in } \mathcal B^+_{\rho_2}(X_0)\setminus\mathcal B^+_{\rho}(X_0).
\end{equation*}
Therefore, for $0<\rho<\rho_1(\rho_2)$, one has
\begin{equation*}
U_{X_0,\rho}\le U\mbox{ in }\mathcal B^+_{\rho_2}(X_0)\setminus\mathcal B^+_{\rho}(X_0).
\end{equation*}
This completes the proof.
\end{proof}

\begin{lemma}\label{l:xigl2}
Let $\rho_1$ be the number given in the proof of Lemma \ref{l:l2small}, then there exists $\rho_0\in (0,\rho_1)$ such that, for all $\rho\in (0,\rho_0)$ and $\xi\in \mathbf R^{n+1}_+\setminus(\bar{\mathcal B}^+_{\rho_2}(X_0)\cup \{0\})$, it holds that
\begin{equation*}
U_{X_0,\rho}(\xi)\le U(\xi).
\end{equation*}
\end{lemma}
\begin{proof}
Set
\begin{equation*}
\varphi(\xi)=\left(\frac{\rho_2}{|\xi-X_0|}\right)^{n-2s}\inf_{\partial''\mathcal B_{\rho_2}(X_0)}U,
\end{equation*}
where $\rho_2$ is given by Lemma \ref{l:l2small}. A direct calculation yields that
\begin{equation*}
\left\{\begin{array}{ll}
         {\rm div}(t^{1-2s}\nabla \varphi)=0, & \mbox{ in }\mathbf R^{n+1}_+\setminus\mathcal B^+_{\rho_2}(X_0), \\
         \displaystyle-\lim_{t\to 0^+}\partial_t\varphi(x,t)=0, & \mbox{ on }\mathbf R^n\setminus \overline{B}_{\rho_2}(x_0).
       \end{array}
\right.
\end{equation*}
From Proposition \ref{p:pn0}, for $\rho_2>0$ small enough, there exists a positive number $\varepsilon_0$ such that, for all $\varepsilon\in (0,\varepsilon_0)$,
\begin{equation*}
U(\xi)\ge \varphi(\xi),\quad \mbox{ in }\overline{\mathcal B}^+_{\varepsilon}(0)\setminus \{0\}.
\end{equation*}
Particularly,
\begin{equation*}
U(\xi)\ge \varphi(\xi),\quad \mbox{on }\partial''\mathcal B^+_{\varepsilon}(0).
\end{equation*}
Note that $\varphi(\xi)\le U(\xi)$ on $\partial''\mathcal B^+_{\rho_2}(X_0)$, then by Lemma \ref{l:non-negative}, we have
\begin{equation*}
U(\xi)\ge \varphi(\xi), \quad \mbox{in }\mathbf R^{n+1}_+\setminus[\mathcal B^+_{\rho_2}(X_0)\cup\mathcal B^+_{\varepsilon}(0)].
\end{equation*}
Define
\begin{equation*}
\rho_0=\min\left\{\rho_1,\rho_2\left(\frac{\inf_{\partial''\mathcal B_{\rho_2}(X_0)}U}{\sup_{\mathcal B_{\rho_2}(X_0)}U}\right)^{\frac{1}{n-2s}}\right\}.
\end{equation*}
Then in $\mathbf R^{n+1}_+\setminus[\mathcal B^+_{\rho_2}(X_0)\cup\mathcal B^+_{\varepsilon}(0)]$, one has
\begin{eqnarray*}
U_{X_0,\rho}(\xi)&=&\left(\frac{\rho}{|\xi-X_0|}\right)^{n-2s}U\left(X_0+\frac{\rho^2(\xi-X_0)}{|\xi-X_0|^2}\right)\notag\\
&\le& \left(\frac{\rho_0}{|\xi-X_0|}\right)^{n-2s}\sup_{\mathcal B_{\rho_2}(X_0)}U\notag\\
&\le&\left(\frac{\rho_2}{|\xi-X_0|}\right)^{n-2s}\inf_{\partial''\mathcal B_{\rho_2}(X_0)}U\le U(\xi).
\end{eqnarray*}
Since $\varepsilon\in (0,\varepsilon_0)$ is arbitrarily, it then holds that
\begin{equation*}
U_{X_0,\rho}(\xi)\le U(\xi),\quad \mbox{ in }\mathbf R^{n+1}_+\setminus(\overline{\mathcal B}^+_{\rho_2}(X_0)\cup \{0\}).
\end{equation*}
This completes the proof.
\end{proof}
\begin{proof}[Proof of Proposition \ref{p:beginning}]
Combining Lemma \ref{l:l2small} and Lemma \ref{l:xigl2}, we get the desired result.
\end{proof}

Define
\begin{equation*}
\bar\rho(X_0)=\sup\{0<\mu\le|X_0|: U_{X_0,\rho}\le U\,\,\mbox{ in }\mathbf R^{n+1}_+\setminus \left(\mathcal B^+_{\rho}(X_0)\cup\{0\}\right),\,\forall\, 0<\rho<\mu\}.
\end{equation*}
By Proposition \ref{p:beginning}, it holds that $\bar\rho(X_0)>0$.
\begin{prop}\label{p:n+1symmetry}
For all $X_0\in\mathbf R^{n+1}_+\setminus \{0\}$, we have
\begin{equation*}
\bar\rho(X_0)=|X_0|.
\end{equation*}
\end{prop}
\begin{proof}
We argue by contradiction. Suppose $\bar\rho(X_0)<|X_0|$. For simplicity, we set $\bar\rho:=\bar\rho(X_0)$. By the definition of $\bar\rho(X_0)$, it holds that
\begin{equation*}
U_{X_0,\bar\rho}\le U, \quad \mbox{ in }\mathbf R^{n+1}_+\setminus(\bar{\mathcal B}^+_{\bar\rho}(X_0)\cup \{0\}).
\end{equation*}
By Kelvin transformation, this is equivalent to
\begin{equation*}
U_{X_0,\bar\rho}\ge U,\quad \mbox{ in }\mathcal B^+_{\bar\rho}(X_0)\setminus\{X_0,0_{\bar\rho, X_0}\}.
\end{equation*}
Here $0_{\bar\rho, X_0}$ is the Kelvin transformation in $\mathbf{R}^{n + 1}_+$ of $0$. Let $\delta\in \left(0,\frac{1}{2}\min\{\bar\rho,\bar\rho-\frac{\bar\rho^2}{|X_0|}\}\right)$,
then, by (\ref{e:mail-extention-c}) and (\ref{e:kte}), $U_{X_0,\bar\rho}-U$ satisfies
\begin{equation*}
\begin{array}{ll}
         {\rm div}(t^{1-2s}(U_{X_0,\bar\rho}-U))=0\mbox{ in } \mathcal B^+_{\bar\rho}(X_0)\setminus [\mathcal B^+_{\delta}(X_0)\cup \mathcal B^+_{\delta}(0_{\bar\rho, X_0})]
       \end{array}
\end{equation*}
and
\begin{eqnarray*}
&&\displaystyle-\lim_{t\to 0^+}t^{1-2s}\partial_t (U_{X_0,\rho}-U)\notag\\
&&=\left(\frac{\rho}{|x-x_0|}\right)^{4s}\frac{\lambda}{|x_{\rho,x_0}|^{2s}}U_{X_0,\rho}(x_{\rho,x_0},0)
+U_{X_0,\rho}^{\frac{n+2s}{n-2s}}(x_{\rho,x_0},0)\notag\\
&&\quad -\left(\frac{\lambda}{|x|^{2s}}U(x,0)+U^{\frac{n+2s}{n-2s}}(x,0)\right)\notag\\
&&=\left[\left(\frac{\rho}{|x-x_0|}\right)^{4s}\frac{\lambda}{|x_{\rho,x_0}|^{2s}}U_{X_0,\rho}(x_{\rho,x_0},0)-\frac{\lambda}{|x|^{2s}}U(x,0)\right]\notag\\
&&\quad +\left[U_{X_0,\rho}^{\frac{n+2s}{n-2s}}(x_{\rho,x_0},0)-U^{\frac{n+2s}{n-2s}}(x,0)\right]\notag\\
&&:=T_3+T_4,\quad \mbox{ on }\partial'[\mathcal B^+_{\bar\rho}(X_0)\setminus (\mathcal B^+_{\delta}(X_0)\cup \mathcal B^+_{\delta}(0_{\bar\rho, X_0}))].
\end{eqnarray*}
Here $x_{\rho,x_0}$ is given by (\ref{e:xxl}).
Note that $T_4\ge 0$ on $\partial'[\mathcal B^+_{\bar\rho}(X_0)\setminus (\mathcal B^+_{\delta}(X_0)\cup \mathcal B^+_{\delta}(0_{\bar\rho, X_0}))]$.
So we only have to estimate $T_3$.

By Lemma \ref{l:kxx0l}, for $x\in B_{\bar\rho}(x_0)\setminus (B_{\delta}(x_0)\cup B_{\delta}(0_{\bar\rho, x_0}))$, it holds that
\begin{equation}\label{e:xkx0bl}
\left(\frac{\bar\rho}{|x-x_0|}\right)^{4s}\frac{\lambda}{|x_{\bar\rho,x_0}|^{2s}}\ge \frac{\lambda}{|x|^{2s}}.
\end{equation}
Then we have, in $\partial'[\mathcal B^+_{\bar\rho}(X_0)\setminus (\mathcal B^+_{\delta}(X_0)\cup \mathcal B^+_{\delta}(0_{\bar\rho, X_0}))]$,
\begin{eqnarray*}
T_3&\ge& \frac{\lambda}{|x|^{2s}}(U_{X_0,\rho}(x_{\rho,x_0},0)-U(x,0))\ge 0.
\end{eqnarray*}
Therefore,
\begin{equation*}
\displaystyle-\lim_{t\to 0^+}t^{1-2s}\partial_t (U_{X_0,\rho}-U)\ge 0\mbox{ in }\partial'[\mathcal B^+_{\bar\rho}(X_0)\setminus (\mathcal B^+_{\delta}(X_0)\cup \mathcal B^+_{\delta}(0_{\bar\rho, X_0}))].
\end{equation*}
If $U_{X_0,\rho}-U$ is not identically $0$, then by Lemma \ref{p:harnack_1}, we have
\begin{equation}\label{e:gtbl}
U_{X_0,\rho}>U,\mbox{ in } \overline{\mathcal B}^+_{\bar\rho}(X_0)\setminus [\mathcal B^+_{\delta}(X_0)\cup \mathcal B^+_{\delta}(0_{\bar\rho, X_0})\cup \partial''\mathcal B^+_{\bar\rho}(X_0)].
\end{equation}
From Proposition \ref{p:pn0}, it holds that
\begin{equation*}
\liminf_{\xi\to X_0}(U_{X_0,\bar\rho}(\xi)-U(\xi))>0
\end{equation*}
and
\begin{equation*}
\liminf_{\xi\to 0_{\bar\rho,X_0}}(U_{X_0,\bar\rho}(\xi)-U(\xi))>0.
\end{equation*}
Hence there exist $\varepsilon_1>0$ and $\delta_1>0$ such that
\begin{equation*}
U_{X_0,\bar\rho}(\xi)>U(X_0)+\varepsilon_1,\quad \forall\, 0<|\xi-X_0|<\delta_1
\end{equation*}
and
\begin{equation*}
U_{X_0,\bar\rho}(\xi)>U(0_{\bar\rho,X_0})+\varepsilon_1,\quad \forall\, 0<|\xi-0_{\bar\rho,X_0}|<\delta_1.
\end{equation*}
Then for all $0<|\xi-X_0|<\delta_1$ and $\bar\rho<\rho<|X_0|$, we have
\begin{eqnarray*}
U_{X_0,\rho}(\xi)&=&\left(\frac{\rho}{|\xi-X_0|}\right)^{n-2s}U\left(X_0+\frac{\rho^2(\xi-X_0)}{|\xi-X_0|^2}\right)\notag\\
&=&\left(\frac{\rho}{\bar\rho}\right)^{n-2s}\left(\frac{\bar\rho}{|\xi-X_0|}\right)^{n-2s}
U\left(X_0+\frac{\bar\rho^2\left(\frac{\bar\rho^2}{\rho^2}(\xi-X_0)\right)}{\left|\frac{\bar\rho^2}{\rho^2}(\xi-X_0)\right|^2}\right)\notag\\
&=&\left(\frac{\rho}{\bar\rho}\right)^{n-2s}U_{X_0,\bar\rho}\left(\frac{\bar\rho^2}{\rho^2}(\xi-X_0)\right)\notag\\
&\ge&U_{X_0,\bar\rho}\left(\frac{\bar\rho^2}{\rho^2}(\xi-X_0)\right)>U(X_0)+\varepsilon_1.
\end{eqnarray*}
Similarly, for all $0<|\xi-0_{\bar\rho,X_0}|<\delta_1$ and $\bar\rho<\rho<|X_0|$,
\begin{eqnarray*}
U_{X_0,\rho}(\xi)\ge U(0_{\bar\rho,X_0})+\varepsilon_1.
\end{eqnarray*}
Since $U$ is uniformly continuous in $\overline{\mathcal B}^+_{\bar\rho}(X_0)$, there exists $\delta_2\in (0,\delta_1/2)$ such that for all $|\xi_1-\xi_2|<\delta_2$ with $\xi_1,\,\xi_2\in \overline{\mathcal B}^+_{\delta_1}(0_{\bar\rho,X_0})$ or $\xi_1,\,\xi_2\in \overline{\mathcal B}^+_{\delta_1}(X_0)$,
$$|U(\xi_1)-U(\xi_2)|<\frac{\varepsilon_1}{2}.$$
Note that we have assume $\bar\rho<|X_0|$, so there exists $\delta_3>0$ such that for all $\bar\rho<\rho<\bar\rho+\delta_3<|X_0|$,
\begin{equation*}
|0_{\rho,X_0}-0_{\bar\rho,X_0}|<\frac{\delta_2}{2}.
\end{equation*}
Thus $0_{\rho,X_0}\in \mathcal B^+_{\frac{\delta_2}{2}}(0_{\bar\rho,X_0})\subset\mathcal B^+_{\delta_2}(0_{\bar\rho,X_0})$.
Hence for all $\xi\in\mathcal B^+_{\delta_2}(0_{\bar\rho,X_0})\setminus\{0_{\rho,X_0}\}$ and $\bar\rho<\rho<\bar\rho+\delta_3$, we have
\begin{equation*}
U_{X_0,\rho}(\xi)\ge U(\xi)+\frac{\varepsilon_1}{2}>U(\xi).
\end{equation*}
Similarly, we obtain that for all $\xi\in\mathcal B^+_{\delta_2}(X_0)\setminus\{X_0\}$ and $\bar\rho<\rho<\bar\rho+\delta_3$,
\begin{equation*}
U_{X_0,\rho}(\xi)\ge U(\xi)+\frac{\varepsilon_1}{2}>U(\xi).
\end{equation*}
Therefore, it holds that for all $\bar\rho<\rho<\bar\rho+\delta_3$,
\begin{equation*}
U_{X_0,\rho}>U \mbox{ in }\left[\mathcal B^+_{\delta_2}(X_0)\setminus\{X_0\}\right]\cup
\left[\mathcal B^+_{\delta_2}(0_{\bar\rho,X_0})\setminus\{0_{\rho,X_0}\}\right].
\end{equation*}

Let $\delta>0$ be a small number which will be fixed later. Denote
\begin{equation*}
K_{\delta,\delta_2}:=\overline{\mathcal B}^+_{\bar\rho-\delta}(X_0)\setminus \left[\mathcal B^+_{\delta_2}(X_0)\cup \mathcal B^+_{\delta_2}(0_{\bar\rho,X_0})\right].
\end{equation*}
From (\ref{e:gtbl}) and the compactness of $K_{\delta,\delta_2}$, there exists a positive $\varepsilon_3$ depending on $\delta$ and $\delta_2$ such that
\begin{equation*}
U_{X_0,\bar\rho}-U>\varepsilon_3 \mbox{ in }K_{\delta,\delta_2}.
\end{equation*}
Since $U$ is uniformly continuous in the compact set $K_{\delta,\delta_2}$, there is a constant $\delta_4$ satisfying $0 < \delta_4 <\delta_3$ such that for all $\bar\rho<\rho<\bar\rho+\delta_4$,
\begin{equation*}
U_{X_0,\rho}-U_{X_0,\bar\rho}>-\frac{\varepsilon_3}{2}\mbox{ in }K_{\delta,\delta_2}.
\end{equation*}
Hence
\begin{equation}\label{e:5.51}
U_{X_0,\rho}-U>\frac{\varepsilon_3}{2}\mbox{ in }K_{\delta,\delta_2}.
\end{equation}

Finally, we prove that $U_{X_0,\rho}>U$ in $\overline{\mathcal B}^+_{\rho}(X_0)\setminus\mathcal B^+_{\bar\rho-\delta}(X_0)$.
In fact, the proof is based on a narrow domain technique which is the same as that in Lemma \ref{l:l2small}. the function $U -U_{X_0,\rho}$ satisfies
\begin{equation}\label{e:divulambda-u1}
{\rm div}\left(t^{1-2s}\nabla (U -U_{X_0,\rho})\right)=0\mbox{ in }\mathcal B^+_{\rho}(X_0)\setminus \overline{\mathcal B}^+_{\bar\rho-\delta}(X_0)
\end{equation}
and
\begin{eqnarray}\label{e:b-ulambda-u1}
-\lim_{t\to 0^+}\partial_t(U -U_{X_0,\rho})
&=&\left[\frac{\lambda}{|x|^{2s}}U(x,0)-\left(\frac{\rho}{|x-x_0|}\right)^{4s}\frac{\lambda}{|x_{\rho,x_0}|^{2s}}U_{X_0,\rho}(x_{\rho,x_0},0)\right]\\
&&+\left[U^{\frac{n+2s}{n-2s}}(x,0)-U_{X_0,\rho}^{\frac{n+2s}{n-2s}}(x_{\rho,x_0},0)\right]\mbox{ on }\partial'\left[\mathcal B^+_{\rho}(X_0)\setminus \overline{\mathcal B}^+_{\bar\rho-\delta}(X_0)\right].\notag
\end{eqnarray}
Let $(U -U_{X_0,\rho})^+:=\max(0,U- U_{X_0,\rho})$ which equals to $0$ on $\partial''\left[\mathcal B^+_{\rho}(X_0)\setminus \overline{\mathcal B}^+_{\bar\rho-\delta}(X_0)\right]$ by (\ref{e:5.51}). We choose $(U - U_{X_0,\rho})^+$ as a test function in the definition of weak solution for (\ref{e:extention_bounded}).
Thus from (\ref{e:divulambda-u1}), (\ref{e:b-ulambda-u1}) and Definition \ref{d:weak-solution}, we have
\begin{eqnarray*}
&&\int_{\mathcal B^+_{\rho}(X_0)\setminus \overline{\mathcal B}^+_{\bar\rho-\delta}(X_0)}t^{1-2s}|\nabla(U -U_{X_0,\rho})^+|^2\\
&&=\int_{\partial'\left[\mathcal B^+_{\rho}(X_0)\setminus \overline{\mathcal B}^+_{\bar\rho-\delta}(X_0)\right]}
\left[\frac{\lambda}{|x|^{2s}}U(x,0)-\left(\frac{\rho}{|x-x_0|}\right)^{4s}\frac{\lambda}{|x_{\rho,x_0}|^{2s}}U_{X_0,\rho}(x_{\rho,x_0},0)\right]
(U -U_{X_0,\rho})^+\notag\\
&&\quad +\int_{\partial'\left[\mathcal B^+_{\rho}(X_0)\setminus \overline{\mathcal B}^+_{\bar\rho-\delta}(X_0)\right]}\left[U^{\frac{n+2s}{n-2s}}(x,0)-U_{X_0,\rho}^{\frac{n+2s}{n-2s}}(x_{\rho,x_0},0)\right](U -U_{X_0,\rho})^+\notag\\
&&:=T_5+T_6.\notag
\end{eqnarray*}
From Lemma \ref{l:kxx0l}, we have, in $\partial'\left[\mathcal B^+_{\rho}(X_0)\setminus \overline{\mathcal B}^+_{\bar\rho-\delta}(X_0)\right]$,
$$\left(\frac{\rho}{|x-x_0|}\right)^{4s}\frac{\lambda}{|x_{\rho,x_0}|^{2s}}\ge \frac{\lambda}{|x|^{2s}}.$$
Then
\begin{eqnarray*}
T_5&\le& \int_{\partial'\left[\mathcal B^+_{\rho}(X_0)\setminus \overline{\mathcal B}^+_{\bar\rho-\delta}(X_0)\right]}\frac{\lambda}{|x|^{2s}}\left((U -U_{X_0,\rho})^+\right)^2\notag\\
&\le& \left(\int_{\partial'\left[\mathcal B^+_{\rho}(X_0)\setminus \overline{\mathcal B}^+_{\bar\rho-\delta}(X_0)\right]}\frac{\lambda^{\frac{n}{2s}}}{|x|^{n}}\right)^{\frac{2s}{n}}
\left(\int_{\partial'\left[\mathcal B^+_{\rho}(X_0)\setminus \overline{\mathcal B}^+_{\bar\rho-\delta}(X_0)\right]}\left((U -U_{X_0,\rho})^+\right)^{\frac{2n}{n-2s}}\right)^{\frac{n-2s}{n}}.
\end{eqnarray*}
On the other hand, by the mean value theorem and H\"older inequality, we have
\begin{eqnarray*}
T_6&\le&C\int_{\partial'\left[\mathcal B^+_{\rho}(X_0)\setminus \overline{\mathcal B}^+_{\bar\rho-\delta}(X_0)\right]} U^{\frac{4s}{n-2s}}\left((U -U_{X_0,\rho}\right)^+)^2\notag\\
&\le& C\left(\int_{\partial'\left[\mathcal B^+_{\rho}(X_0)\setminus \overline{\mathcal B}^+_{\bar\rho-\delta}(X_0)\right]} U^{\frac{2n}{n-2s}}\right)^{\frac{2s}{n}}\left(\int_{\partial'\left[\mathcal B^+_{\rho}(X_0)\setminus \overline{\mathcal B}^+_{\bar\rho-\delta}(X_0)\right]}\left((U -U_{X_0,\rho})^+\right)^{\frac{2n}{n-2s}}\right)^{\frac{n-2s}{n}}.\notag
\end{eqnarray*}
Hence
\begin{eqnarray*}
T_5+T_6&\le&C\left[\left(\int_{\partial'\left[\mathcal B^+_{\rho}(X_0)\setminus \overline{\mathcal B}^+_{\bar\rho-\delta}(X_0)\right]}\frac{\lambda^{\frac{n}{2s}}}{|x|^{n}}\right)^{\frac{2s}{n}}+\left(\int_{\partial'\left[\mathcal B^+_{\rho}(X_0)\setminus \overline{\mathcal B}^+_{\bar\rho-\delta}(X_0)\right]} U^{\frac{2n}{n-2s}}\right)^{\frac{2s}{n}}\right]\notag\\
&&\,\,\times\int_{\mathcal B^+_{\rho}(X_0)\setminus \overline{\mathcal B}^+_{\bar\rho-\delta}(X_0)}t^{1-2s}|\nabla(U -U_{X_0,\rho})^+|^2.
\end{eqnarray*}
Let $\delta$ and $\delta_3$ be sufficiently small such that
\begin{equation*}
C\left[\left(\int_{\partial'\left[\mathcal B^+_{\rho}(X_0)\setminus \overline{\mathcal B}^+_{\bar\rho-\delta}(X_0)\right]}\frac{\lambda^{\frac{n}{2s}}}{|x|^{n}}\right)^{\frac{2s}{n}}+\left(\int_{\partial'\left[\mathcal B^+_{\rho}(X_0)\setminus \overline{\mathcal B}^+_{\bar\rho-\delta}(X_0)\right]} U^{\frac{2n}{n-2s}}\right)^{\frac{2s}{n}}\right]<\frac{1}{2}.
\end{equation*}
Then
\begin{equation*}
\nabla(U -U_{X_0,\rho})^+=0 \quad\mbox{ in } \mathcal B^+_{\rho}(X_0)\setminus\mathcal B^+_{\bar\rho-\delta}(X_0).
\end{equation*}
Since $(U -U_{X_0,\rho})^+ = 0$ on $\partial''\left[\mathcal B^+_{\rho}(X_0)\setminus \overline{\mathcal B}^+_{\bar\rho-\delta}(X_0)\right]$, we have
\begin{equation*}
(U -U_{X_0,\rho})^+=0\quad\mbox{ in } \mathcal B^+_{\rho}(X_0)\setminus\mathcal B^+_{\bar\rho-\delta}(X_0).
\end{equation*}
So $U_{X_0,\rho}>U$ in $\overline{\mathcal B}^+_{\rho}(X_0)\setminus\mathcal B^+_{\bar\rho-\delta}(X_0)$.

Putting the results obtained above together, we have that
\begin{equation*}
U_{X_0,\rho}\ge U, \quad\mbox{ in } \mathcal B^+_{\rho}(X_0)\setminus\{X_0,0_{\bar\rho, X_0}\}.
\end{equation*}
This contradicts with the definition of $\bar\rho$.
\end{proof}

\begin{proof}[Proof of Theorem \ref{t:regulatiry} (2)]
Let $\xi=(x,t)$. From Proposition \ref{p:n+1symmetry}, we have that for $X_0=(x_0,0)$,
\begin{equation}\label{e:symmetry-proof}
U_{X_0,\rho}(\xi)\le U(\xi), \quad\forall\, |\xi-X_0|\ge\rho,\,\xi\ne 0,\,\forall\, 0<\rho<|x_0|.
\end{equation}
Let ${\bf e}$ be any unit vector in $\mathbf R^n$, for any $l>0$, $\xi\in \mathbf R^{n+1}_+$ with $\langle(\xi-l{\bf e}), {\bf e}\rangle<0$. Choosing $x_0=R{\bf e}$ and $\rho=R-l$ in (\ref{e:symmetry-proof}) and letting $R\to+\infty$, we obtain that
\begin{equation*}
U(x,t)\ge U(x-2(\langle x,{\bf e}\rangle-l){\bf e},t).
\end{equation*}
This yields the conclusion (2) of Theorem \ref{t:regulatiry}.
\end{proof}

\begin{proof}[Proof of Corollary \ref{c:fractional-symmetry}]
The first conclusion is a direct corollary of Theorem \ref{t:regulatiry}. We now focus on the second conclusion.

From Theorem \ref{t:regulatiry} and Caffarelli-Silvestre extension, we have that the nonnegative solutions to (\ref{e:main3}) is radially symmetric about the origin and non-increasing in radial directions. In what follows, for simplicity, we use the notation $u(r)$ instead of $u(x)$ for $r=|x|$. Thus, we only need to prove that $u$ is strictly decreasing in radial directions if $u$ is nontrivial.

We argue by contradiction. Assume that there exist $0<r_1<r_2<\infty$ such that $u(r_1)=u(r_2)$. Without loss of generality, we may assume that $u(r)>u(r_1)$ for $r<r_1$ and $u(r)<u(r_2)$ for $r>r_2$. Let $A_{r_1,r_2}:=\{x\in \mathbf R^n\,|\,r_1<|x|<r_2\}$. So $u$ is a constant in $A_{r_1,r_2}$. Let $x_0$ be any point in $A_{r_1,r_2}$. Note that (\ref{e:main3}) at $x_0$ is
\begin{equation}\label{e:decrreasing-x0}
C_{n,s}{\rm P.V.}\int_{\mathbf R^n}\frac{u(x_0)-u(y)}{|x_0-y|^{n+2s}}dy=\frac{\lambda}{|x_0|^{2s}}u(x_0)+u^{\frac{n+2s}{n-2s}}(x_0).
\end{equation}
A calculation yields that
\begin{eqnarray}\label{e:fractionalL}
&&C_{n,s}{\rm P.V.}\int_{\mathbf R^n}\frac{u(x_0)-u(y)}{|x_0-y|^{n+2s}}dy\notag\\
&&=C_{n,s}\int_{\mathbf R^n\setminus \overline{A}_{r_1,r_2}}\frac{u(x_0)-u(y)}{|x_0-y|^{n+2s}}dy\notag\\
&&=C_{n,s}\int_{B_{r_1}}\frac{u(x_0)-u(y)}{|x_0-y|^{n+2s}}dy+C_{n,s}\int_{\mathbf R^n\setminus \overline{B}_{r_2}}\frac{u(x_0)-u(y)}{|x_0-y|^{n+2s}}dy.
\end{eqnarray}
Let $\eta_1=(z_1,0,\cdots,0)$, $\eta_2=(z_2,0,\cdots,0)$ with $r_1<z_1<z_2<r_2$. Observe that
\begin{equation*}
u(\eta_1)-u(y)=u(\eta_2)-u(y)<0,\quad\forall y\in B_{r_1}
\end{equation*}
and
\begin{equation*}
\frac{1}{|\eta_1-y|^{n+2s}}>\frac{1}{|\eta_2-y|^{n+2s}},\quad\forall y\in B_{r_1}.
\end{equation*}
Then we have
\begin{equation}\label{e:eta-1}
C_{n,s}\int_{B_{r_1}}\frac{u(\eta_1)-u(y)}{|\eta_1-y|^{n+2s}}dy<C_{n,s}\int_{B_{r_1}}\frac{u(\eta_2)-u(y)}{|\eta_2-y|^{n+2s}}dy.
\end{equation}

Next, we prove
\begin{equation}\label{e:eta-2}
C_{n,s}\int_{\mathbf R^n\setminus \overline{B}_{r_2}}\frac{u(\eta_1)-u(y)}{|\eta_1-y|^{n+2s}}dy<C_{n,s}\int_{\mathbf R^n\setminus \overline{B}_{r_2}}\frac{u(\eta_2)-u(y)}{|\eta_2-y|^{n+2s}}dy.
\end{equation}

In fact, let $y=(y_1,\cdots,y_n)$ and
$$I(\eta)=\int_{\mathbf R^n\setminus \overline{B}_{r_2}}\frac{u(\eta)-u(y)}{|\eta-y|^{n+2s}}dy.$$
Set
$$P=\left\{y=(y_1,\cdots,y_n)\in \mathbf R^n\,|\,y_1=\frac{z_1+z_2}{2}\right\},$$
that is, $P$ is the hyperplane orthogonal segment $[\eta_1,\eta_2]$ at $\frac{\eta_1+\eta_2}{2}$.
Define
$$D_1=\left\{y\in \mathbf R^n\setminus \overline{B}_{r_2}(0)\,|\, y_1>\frac{z_1+z_2}{2}\right\}$$
and
$$D_2=\left\{y\in \mathbf R^n\setminus \overline{B}_{r_2}(0)\,|\, y_1<\frac{z_1+z_2}{2}\right\}.$$
A direct computation yields that
\begin{eqnarray*}
I(\eta_2)-I(\eta_1)&=&\int_{D_1}\left(\frac{1}{|\eta_2-y|^{n+2s}}-\frac{1}{|\eta_1-y|^{n+2s}}\right)\left(u(\eta_1)-u(y)\right)dy\notag\\
&&+\int_{D_2}\left(\frac{1}{|\eta_2-y|^{n+2s}}-\frac{1}{|\eta_1-y|^{n+2s}}\right)\left(u(\eta_1)-u(y)\right)dy\notag\\
&:=& T_7+T_8.
\end{eqnarray*}
Since
\begin{equation*}
u(\eta_1)-u(y)=u(\eta_2)-u(y)>0,\quad\forall y\in \mathbf R^n\setminus \overline{B}_{r_2},
\end{equation*}
\begin{equation*}
\frac{1}{|\eta_1-y|^{n+2s}}<\frac{1}{|\eta_2-y|^{n+2s}},\quad\forall y\in D_1
\end{equation*}
and
\begin{equation*}
\frac{1}{|\eta_1-y|^{n+2s}}>\frac{1}{|\eta_2-y|^{n+2s}},\quad\forall y\in D_2,
\end{equation*}
we have
\begin{equation*}
T_7>0 \quad\mbox{and}\quad T_8<0.
\end{equation*}

We now compare the absolute values of $T_7$ and $T_8$. Let $\hat y$ be the reflection point of $y$ with respect to the hyperplane $P$, so $\hat 0=(z_1+z_2,0,\cdots,0)$. Define $\hat D_2=D_1\setminus \overline{B}_{r_2}(\hat 0)$,
that is, $\hat D_2$ is the reflection domain of $D_2$ with respect to $P$. Since $u$ is non-increasing in radial direction, we have
\begin{equation*}
u(\hat y)\le u(y),\quad \forall y\in D_2.
\end{equation*}
Note that $\hat\eta_1=\eta_2$, then we obtain that
\begin{eqnarray*}
T_8
&=&\int_{D_2}\left(\frac{1}{|\eta_1- \hat y|^{n+2s}}-\frac{1}{|\eta_2-\hat y|^{n+2s}}\right)\left(u(\eta_2)-u(y)\right)dy\notag\\
&=&\int_{\hat D_2}\left(\frac{1}{|\eta_1- y|^{n+2s}}-\frac{1}{|\eta_2-y|^{n+2s}}\right)\left(u(\eta_2)-u(\hat y)\right)dy\notag\\
&\ge &\int_{\hat D_2}\left(\frac{1}{|\eta_1- y|^{n+2s}}-\frac{1}{|\eta_2-y|^{n+2s}}\right)\left(u(\eta_2)-u(y)\right)dy\notag\\
&=&-\int_{\hat D_2}\left(\frac{1}{|\eta_2- y|^{n+2s}}-\frac{1}{|\eta_1-y|^{n+2s}}\right)\left(u(\eta_2)-u(y)\right)dy.
\end{eqnarray*}
By the definition of $\hat D_2$, we get
\begin{eqnarray*}
T_7+T_8\ge \int_{D_1\cap B_{r_2}(\hat 0)}\left(\frac{1}{|\eta_2-y|^{n+2s}}-\frac{1}{|\eta_1-y|^{n+2s}}\right)\left(u(\eta_1)-u(y)\right)dy
>0.
\end{eqnarray*}
Therefore, (\ref{e:eta-2}) holds.

Finally, from (\ref{e:fractionalL}), (\ref{e:eta-2}) and (\ref{e:eta-1}), we have
\begin{equation}\label{e:c1}
C_{n,s}{\rm P.V.}\int_{\mathbf R^n}\frac{u(\eta_1)-u(y)}{|\eta_1-y|^{n+2s}}dy<C_{n,s}{\rm P.V.}\int_{\mathbf R^n}\frac{u(\eta_2)-u(y)}{|\eta_2-y|^{n+2s}}dy.
\end{equation}
On the other hand, it holds that
\begin{equation*}
\frac{\lambda}{|\eta_1|^{2s}}u(\eta_1)+u^{\frac{n+2s}{n-2s}}(\eta_1)\ge \frac{\lambda}{|\eta_2|^{2s}}u(\eta_2)+u^{\frac{n+2s}{n-2s}}(\eta_2).
\end{equation*}
This is impossible since inequality (\ref{e:c1}) and equation (\ref{e:decrreasing-x0}) hold at $\eta_1$ and $\eta_2$. Therefore, the assumption at the beginning can not be true.
That is, $u(r)$ is strictly decreasing with respect to $r$. This completes the proof.
\end{proof}


\section{Proof of theorem \ref{t:lglne}}\label{s:none}

In this section, we shall prove Theorem \ref{t:lglne}.

Define
\begin{equation*}
Q(U,V)=\int_{\mathbf R^{n+1}_+}t^{1-2s}\nabla U\cdot \nabla V-\kappa_s\lambda\int_{\mathbf R^n}\frac{U(x,0)V(x,0)}{|x|^{2s}},\quad U,\,V\in \dot{H}^1(\mathbf R^{n+1}_+,t^{1-2s}).
\end{equation*}
Let
\begin{equation}\label{e:qud}
Q(U):=\int_{\mathbf R^{n+1}_+}t^{1-2s}|\nabla U|^2-\kappa_s\lambda\int_{\mathbf R^n}\frac{U^2(x,0)}{|x|^{2s}},\quad U\in \dot{H}^1(\mathbf R^{n+1}_+,t^{1-2s}).
\end{equation}
Note that by Hardy inequality and trace inequality, we have
\begin{equation}\label{e:hardy-trace}
\kappa_s\Lambda_{n,s}\int_{\mathbf R^n}\frac{U^2(x,0)}{|x|^{2s}}\le \int_{\mathbf R^{n+1}_+}t^{1-2s}|\nabla U|^2.
\end{equation}
Thus, if $\lambda<\Lambda_{n,s}$, then $Q(U)\ge 0$ for all $U\in \dot{H}^1(\mathbf R^{n+1}_+,t^{1-2s})$. By (\ref{e:hardy-trace}) we obtain that
$Q$ is continuous in $\dot{H}^1(\mathbf R^{n+1}_+,t^{1-2s})\times \dot{H}^1(\mathbf R^{n+1}_+,t^{1-2s})$. Then there exists a unique bounded symmetric operator $\mathcal L_Q\in \mathcal L(\dot{H}^1(\mathbf R^{n+1}_+,t^{1-2s}))$ such that 
$$\langle \mathcal L_Q U,V\rangle_{\dot{H}^1(\mathbf R^{n+1}_+,t^{1-2s})}=Q(U,V).$$

Define the first eigenvalue of $Q$ as
\begin{equation*}
\nu_1(\lambda)=\inf_{\begin{array}{c}
                       U\in \dot{H}^1(\mathbf R^{n+1}_+,t^{1-2s}) \\
                       \mbox{ and }U(\cdot,0)\ne 0
                     \end{array}}
\frac{Q(U)}{\kappa_s\int_{\mathbf R^n}\frac{U^2(x,0)}{|x|^{2s}}}.
\end{equation*}
Hardy-type inequality yields that if $\lambda=0$, then $\nu_1(0)=\Lambda_{n,s}$.

Let $\mathbf S^n$ be the unit $n$-dimensional sphere and
\begin{equation*}
\mathbf S^n_+=\{\theta=(\theta_1,\cdots,\theta_n,\theta_{n+1})\in \mathbf S^n:\,\theta_{n+1}>0\}.
\end{equation*}
Thus
$$\mathbf R^{n+1}_+=\mathbf R_+\times \mathbf S^n_+ \mbox{ with }X=(r,\theta):=\left(|X|,\frac{X}{|X|}\right).$$

Let $H^1(\mathbf S^n_+,\theta_{n+1}^{1-2s})$ be the completion of $C^{\infty}(\overline{\mathbf S^n_+})$ with the weighted norm given by
\begin{equation*}
\|\psi\|_{H^1(\mathbf S^n_+,\theta_{n+1}^{1-2s})}=\left(\int_{\mathbf S^n_+}\theta_{n+1}^{1-2s}(|\nabla_{\mathbf S^n}\psi(\theta)|^2+\psi^2(\theta))\right)^{\frac{1}{2}}.
\end{equation*}
Define
\begin{equation*}
L^2(\mathbf S^n_+,\theta_{n+1}^{1-2s}):=\left\{\psi:\mathbf S^n_+\to \mathbf R \mbox{ measurable such that }\int_{\mathbf S^n_+}\theta_{n+1}^{1-2s}\psi^2(\theta)<\infty\right\}.
\end{equation*}
Since the weight $\theta_{n+1}^{1-2s}$ belongs to the second Muckenhoupt class $A_2$, the embedding
$H^1(\mathbf S^n_+,\theta_{n+1}^{1-2s})\hookrightarrow L^2(\mathbf S^n_+,\theta_{n+1}^{1-2s})$ is compact (see \cite{FabesKenigCarlosCPDE1982}).
The trace operator
\begin{equation}\label{e:htrace1}
{\rm tr}:H^1(\mathbf S^n_+,\theta_{n+1}^{1-2s})\to L^2(\mathbf S^{n-1})
\end{equation}
is well-defined and satisfies the following inequality: for every $\psi\in H^1(\mathbf S^n_+,\theta_{n+1})$, one has
\begin{equation}\label{e:htrace2}
\kappa_s\Lambda_{n,s}\left(\int_{\mathbf S^{n-1}}|\psi(\theta',0)|^2\right)
\le \left(\frac{n-2s}{2}\right)^2\int_{\mathbf S^n_+}\theta_{n+1}^{1-2s}\psi^2+\int_{\mathbf S^n_+}\theta_{n+1}^{1-2s}|\nabla_{\mathbf S^n}\psi|^2,
\end{equation}
where $\mathbf S^{n-1}=\partial \mathbf S_+^n=\{(\theta',\theta_{n+1})\,|\,\theta_{n+1}=0\}$.
This trace inequality was obtained in \cite[Lemma 2.2]{FallFelliCPDE2014}.

Consider the following eigenvalue problem
\begin{equation}\label{e:eigenvalue1}
\left\{\begin{array}{rl}
         {\rm div_{\mathbf S^n}}(\theta_{n+1}^{1-2s}\nabla_{\mathbf S^n}\psi)=\left(\frac{n-2s}{2}\right)^2\theta_{n+1}^{1-2s}\psi, & \mbox{ in }\mathbf S^n_+, \\
         -\displaystyle\lim_{\theta_{n+1}\to 0}\theta_{n+1}^{1-2s}\nabla_{\mathbf S^n}\psi\cdot \mathbf e_{n+1}-\lambda\kappa_s\psi=\mu\kappa_s\psi, & \mbox{ on }\partial\mathbf S^n_+,
       \end{array}
\right.
\end{equation}
where $\mathbf e_{n+1}=(0,\cdots,0,1)$.
We say that $\mu\in \mathbf R$ is an eigenvalue of problem (\ref{e:eigenvalue1}) if there exists $\psi\in H^1(\mathbf S^n_+,\theta_{n+1}^{1-2s})\setminus\{0\}$ such that
\begin{equation*}
E(\psi,\phi)=\mu\kappa_s\int_{\mathbf S^{n-1}}\psi\phi,\quad\mbox{for all }\phi\in H^1(\mathbf S^n_+,\theta_{n+1}^{1-2s}),
\end{equation*}
where
\begin{equation*}
E:H^1(\mathbf S^n_+,\theta_{n+1}^{1-2s})\times H^1(\mathbf S^n_+,\theta_{n+1}^{1-2s})\to \mathbf R
\end{equation*}
is given by
\begin{equation*}
E(\psi,\phi)=\int_{\mathbf S^n_+}\theta_{n+1}^{1-2s}\nabla_{\mathbf S^n}\psi\cdot\nabla_{\mathbf S^n}\phi+\left(\frac{n-2s}{2}\right)^2\int_{\mathbf S^n_+}\theta_{n+1}^{1-2s}\psi\phi-\lambda\kappa_s\int_{\mathbf S^{n-1}}\psi\phi.
\end{equation*}
By (\ref{e:htrace1}) and (\ref{e:htrace2}), $E$ is continuous and weakly coercive in $H^1(\mathbf S^n_+,\theta_{n+1}^{1-2s})$. For all $\psi\in H^1(\mathbf S^n_+,\theta_{n+1}^{1-2s})$, $E(\psi):=E(\psi,\psi)$ is the corresponding quadratic form.

Consider the following minimization problem
\begin{equation*}
\mu_1(\lambda)=\inf_{\psi\in H^1(\mathbf S^n_+,\theta_{n+1}^{1-2s})\setminus\{0\}}\frac{E(\psi)}{\kappa_s\int_{\mathbf S^{n-1}}\psi^2}.
\end{equation*}
By the proof of (\ref{e:htrace2}), the constant $\Lambda_{n,s}$ is sharp (see \cite[Lemma 2.2]{FallFelliCPDE2014}). Moreover, from the compactness of $\mathbf S^n_+$ and \cite[Theorem 6.16]{Salsa08}, we deduce that
$\mu_1(\lambda)=\Lambda_s-\lambda$ is achieved at some positive function $\psi_1$ in $H^1(\mathbf S^n_+,\theta_{n+1}^{1-2s})$. Therefore,
\begin{equation}\label{e:sextension}
\left\{\begin{array}{cc}
         {\rm div}_{\mathbf S^n}(\theta_{n+1}^{1-2s}\nabla_{\mathbf S^n}\psi_1)=\left(\frac{2s-n}{2}\right)^2\theta_{n+1}^{1-2s}\psi_1 & \mbox{ in }\mathbf S^n_+,\\
          -\displaystyle\lim_{\theta_{n+1}\to 0}\theta_{n+1}^{1-2s}\nabla_{\mathbf S^n}\psi_1\cdot \mathbf e_{n+1}=\Lambda_{n,s}\kappa_s\psi_1, & \mbox{ on }\partial\mathbf S^n_+.
       \end{array}
\right.
\end{equation}

We now prove
\begin{prop}
For all $\lambda\in \mathbf R$, we have
\begin{equation*}
\mu_1(\lambda)=\nu_1(\lambda).
\end{equation*}
\end{prop}
\begin{proof}
The idea of the proof is a combination of the method in \cite{Terrachini:ADE96} and Caffarelli-Silvestre extension of fractional Laplacian.

Firstly, we prove $\mu_1(\lambda)\ge \nu_1(\lambda)$. Let $\eta:\mathbf R_+\to [0,1]$ be a smooth cut-off function such that
$\eta(l)=0$, for $l\in [0,\frac{1}{2}]$, and $\eta(l)=1$, for $l\in [1,+\infty)$. For $\varepsilon\in (0,1)$, define
\begin{equation*}
\eta_{\varepsilon}(l)=\left\{\begin{array}{ll}
                               \eta(l/\varepsilon), & l\le 1, \\
                               \eta(1/(\varepsilon l)), & l\ge 1.
                             \end{array}
\right.
\end{equation*}
Notice that $\eta_{\varepsilon}(l)=\eta_{\varepsilon}(1/l)$, $\forall\,l>0$. Let $\psi_1\in H^1(\mathbf S^n_+,\theta_{n+1}^{1-2s})$ be a positive eigenfunction associated to $\mu_1(\lambda)$. Define
\begin{equation*}
W_{\varepsilon}(X)=|X|^{(2s-n)/2}\eta_{\varepsilon}(|X|)\psi_1(X/|X|).
\end{equation*}
Obviously, $W_{\varepsilon}$ belongs to $\dot{H}^1(\mathbf R^{n+1}_+,t^{1-2s})$. Moreover, $W_{\varepsilon}(X)=\varepsilon^{\frac{2s-n}{2}}W_1(\frac{X}{\varepsilon})$ when $|X|\le \varepsilon$, and, $W_{\varepsilon}(X)=\varepsilon^{\frac{n-2s}{2}}W_1(\varepsilon X)$ when $|X|\ge \varepsilon^{-1}$. Then we have
\begin{eqnarray*}
&&\int_{\mathcal B^+_{\varepsilon}(0)\cup [\mathbf R^{n+1}_+\setminus\mathcal B^+_{1/\varepsilon}(0)]}t^{1-2s}|\nabla W_{\varepsilon}|^2
+\int_{B_{\varepsilon}(0)\cup [\mathbf R^{n}\setminus B_{1/\varepsilon}(0)]}\frac{W_{\varepsilon}^2}{|X|^{2s}}\notag\\
&&=\int_{\mathbf R^{n+1}_+}t^{1-2s}|\nabla W_{1}|^2+\int_{\mathbf R^n}\frac{W_{1}^2}{|X|^{2s}}\le C,
\end{eqnarray*}
where $C$ is a positive constant independent on $\varepsilon$. Therefore,
\begin{eqnarray*}
\nu_1(\lambda)&\le&\frac{\int_{\mathbf R^{n+1}_+}t^{1-2s}|\nabla W_{\varepsilon}|^2-\kappa_s\lambda\int_{\mathbf R^n}\frac{W_{\varepsilon}^2}{|x|^{2s}}}{\kappa_s\int_{\mathbf R^n}\frac{W_{\varepsilon}^2}{|x|^{2s}}}\\
&\le& \frac{C+\int_{\mathcal B^+_{1/\varepsilon}(0)\setminus\mathcal B^+_{\varepsilon}(0)}t^{1-2s}|\nabla W_{\varepsilon}|^2
-\kappa_s\lambda\int_{ B_{1/\varepsilon}(0)\setminus B_{\varepsilon}(0)}\frac{W_{\varepsilon}^2}{|x|^{2s}}}{\kappa_s\int_{ B_{1/\varepsilon}(0)\setminus B_{\varepsilon}(0)}\frac{W_{\varepsilon}^2}{|x|^{2s}}}.\notag
\end{eqnarray*}
In polar coordinates, we have
\begin{eqnarray*}
&&\int_{\mathcal B^+_{1/\varepsilon}(0)\setminus\mathcal B^+_{\varepsilon}(0)}t^{1-2s}|\nabla W_{\varepsilon}|^2\\
&&=\int_{\varepsilon}^{1/\varepsilon}r^{-1}\int_{\mathbf S^{n+1}_+}\theta_{n+1}^{1-2s}
\left[\left(\frac{2s-n}{2}\right)^2\psi_1(\theta)^2+|\nabla_{\mathbf S^n}\psi_1(\theta)|^2\right]dSdr\notag\\
&&=2\log\varepsilon^{-1}\int_{\mathbf S^{n+1}_+}\theta_{n+1}^{1-2s}
\left[\left(\frac{2s-n}{2}\right)^2\psi_1(\theta)^2+|\nabla_{\mathbf S^n}\psi_1(\theta)|^2\right]dS
\end{eqnarray*}
and
\begin{eqnarray*}
\int_{ B_{1/\varepsilon}(0)\setminus B_{\varepsilon}(0)}\frac{W_{\varepsilon}^2}{|x|^{2s}}
&=&\int_{\varepsilon}^{1/\varepsilon}\int_{\mathbf S^{n-1}}r^{-1}\psi_1(\theta)^2dS'dr=2\log\varepsilon^{-1}\int_{\mathbf S^{n-1}}\psi_1(\theta)^2dS'.
\end{eqnarray*}
Then
\begin{eqnarray*}
\nu_1(\lambda)&\le&\frac{C+2\log\varepsilon^{-1}\left[\int_{\mathbf S^{n+1}_+}\theta_{n+1}^{1-2s}
\left(\left(\frac{2s-n}{2}\right)^2\psi_1^2+|\nabla_{\mathbf S^n}\psi_1|^2\right)-\kappa_2\lambda\int_{\mathbf S^{n-1}}\psi_1^2\right]}{2\kappa_s\log\varepsilon^{-1}\int_{\mathbf S^{n-1}}\psi_1^2}.\notag
\end{eqnarray*}
Therefore, letting $\varepsilon\to 0$, we obtain that $\nu_1(\lambda)\le \mu_1(\lambda)$.

Secondly, we prove the reverse inequality. Let $W\in C_0^{\infty}(\mathbf R^{n+1}_+\setminus\{0\})$, define
\begin{equation}\label{e:wtilde}
\tilde W(X)=\left(\int_0^{\infty}\frac{1}{r^{n+1-2s}}W^2\left(\frac{X}{r}\right)dr\right)^{\frac{1}{2}},
\end{equation}
which is a homogeneous function of degree $(2s-n)/2$. A standard calculation yields that on $\partial \mathbf R^{n+1}_+=\mathbf R^n$,
\begin{equation}\label{e:wtfx}
\int_{\mathbf S^{n-1}}\tilde W^2(\theta',0)dS'=\int_{\mathbf R^n}\frac{W^2(x,0)}{|x|^{2s}}dx.
\end{equation}
In polar coordinates, we have
\begin{eqnarray}\label{e:nwpc}
|\nabla\tilde W(X)|^2&=&\left|\nabla \left(r^{\frac{2s-n}{2}}\tilde W\left(\theta\right)\right)\right|^2\\
&=&\left(\frac{2s-n}{2}\right)^2r^{2s-2-n}\tilde W^2(\theta)+r^{2s-2-n}|\nabla_{\mathbf S^n}\tilde W(\theta)|^2.\notag
\end{eqnarray}
Therefore, choosing $r=1$ in (\ref{e:nwpc}), we get
\begin{eqnarray}\label{e:nsw}
\int_{\mathbf S^n_+}\theta_{n+1}^{1-2s}|\nabla\tilde W(\theta)|^2dS=\int_{\mathbf S^n_+}\theta_{n+1}^{1-2s}\left[\left(\frac{2s-n}{2}\right)^2\tilde W^2(\theta)+|\nabla_{\mathbf S^n}\tilde W(\theta)|^2\right]dS.
\end{eqnarray}
On the other hand, differentiating both sides of (\ref{e:wtilde}) and using H\"older inequality, we have that
\begin{equation*}
|\nabla\tilde W(X)|\le \left(\int_{0}^{\infty}\frac{1}{r^{n+3-2s}}|\nabla W|^2\left(\frac{X}{r}\right)dr\right)^{\frac{1}{2}}.
\end{equation*}
Then it holds that
\begin{eqnarray}\label{e:nw}
\int_{\mathbf S^n_+}\theta_{n+1}^{1-2s}|\nabla\tilde W(\theta)|^2dS
&\le&\int_{\mathbf S^n_+}\theta_{n+1}^{1-2s}\left(\int_{0}^{\infty}\frac{1}{r^{n+3-2s}}|\nabla W|^2\left(\frac{\theta}{r}\right)dr\right)dS\notag\\
&=&-\int_{\mathbf S^n_+}\int_{0}^{\infty}\left(\frac{\theta_{n+1}}{r}\right)^{1-2s}|\nabla W|^2\left(\frac{\theta}{r}\right)\left(\frac{1}{r}\right)^nd\left(\frac{1}{r}\right)dS\notag\\
&=&\int_{\mathbf R^{n+1}_+}t^{1-2s}|\nabla W|^2(X)dX.
\end{eqnarray}
Therefore, by (\ref{e:wtfx}), (\ref{e:nsw}), (\ref{e:nw}) and the definition of $\mu_1(\lambda)$, we obtain that
\begin{eqnarray*}
\mu_1(\lambda)&\le& \frac{\int_{\mathbf S^n_+}\theta_{n+1}^{1-2s}\left(|\nabla_{\mathbf S^n}\tilde W(\theta)|^2+\left(\frac{n-2s}{2}\right)^2\tilde W^2(\theta)\right)dS-\lambda\kappa_s\int_{\mathbf S^{n-1}}\tilde W^2(\theta',0)}{\kappa_s\int_{\mathbf S^{n-1}}\tilde W^2(\theta',0)dS'}\notag\\
&\le&\frac{\int_{\mathbf R^{n+1}_+}t^{1-2s}|\nabla W|^2(X)dX-\lambda\kappa_s\int_{\mathbf R^n}\frac{W^2(x,0)dS'}{|x|^{2s}}dx}{\kappa_s\int_{\mathbf R^n}\frac{W^2(x,0)}{|x|^{2s}}dx}.
\end{eqnarray*}
Note that $C_0^{\infty}(\mathbf R^{n+1}_+\setminus\{0\})$ is dense in $H^1(\mathbf R^{n+1}_+,t^{1-2s})$, thus $\mu_1(\lambda)\le \nu_1(\lambda)$. This completes the proof.
\end{proof}

\begin{proof}[Proof of Theorem \ref{t:lglne}]\label{sb:proof-1.5}
We argue by contradiction. Suppose there exists a nonnegative solution $U\in \dot{H}^{1}(\mathbf R^{n+1}_+,t^{1-2s})$ of (\ref{e:mail-extention-c}). Then from Proposition \ref{p:Harnack},  $U$ is positive. Define
\begin{equation*}
W_1(X)=|X|^{\frac{2s-n}{2}}\psi_1\left(\frac{X}{|X|}\right).
\end{equation*}
By (\ref{e:sextension}) we have
\begin{equation}\label{e:dw1}
 {\rm div}(t^{1-2s}\nabla W_1)=0,\quad \mbox{in } \mathbf R^{n+1}_+.
\end{equation}
A direct calculus yields
\begin{eqnarray*}
\partial_t W_1(x,t)&=&\left(\frac{2s-n}{2}\right)t|X|^{\frac{2s-n-4}{2}}\psi_1\left(\frac{X}{|X|}\right)
-t|X|^{\frac{2s-n-6}{2}}\sum_{j=1}^{n+1}[\partial_j\psi_1]\left(\frac{X}{|X|}\right)\notag\\
&&+|X|^{\frac{2s-n-2}{2}}[\partial_t\psi_1]\left(\frac{X}{|X|}\right).
\end{eqnarray*}
It follows that for $x\in\mathbf R^n\setminus \{0\}$,
\begin{eqnarray}\label{e:boundaryw1}
-\lim_{t\to 0}t^{1-2s}\partial_t W_1(x,t)&=&\Lambda_{n,s}\kappa_s\frac{W_1(x,0)}{|x|^{2s}}.
\end{eqnarray}

Let $\mathcal A_{r,R}^+:=\{X\in\mathbf R^{n+1}_+\,|\,r\le |X|\le R\}$ and $A_{r,R}=\{x\in \mathbf R^n \,|\,r\le |x|\le R\}$. Multiplying the first equation of (\ref{e:mail-extention-c}) by $ W_1$, integrating the terms in $\mathcal A_{r,R}^+$ and by the divergence theorem, we obtain that
\begin{eqnarray}\label{e:uwd}
\int_{\mathcal A_{r,R}^+}t^{1-2s}\nabla U\cdot\nabla W_1&=&\int_{\mathbf S^n_{R,+}} W_1t^{1-2s}\partial_r U-\int_{\mathbf S^n_{r,+}} W_1t^{1-2s}\partial_r U\\
&&+\kappa_{s}\int_{A_{r,R}}\frac{\lambda}{|x|^{2s}}W_1(x,0)U(x,0)+W_1(x,0) U^{\frac{n+2s}{n-2s}}(x,0).\notag
\end{eqnarray}
Similarly, multiplying (\ref{e:dw1}) by $U$ and using (\ref{e:boundaryw1}) yield that
\begin{eqnarray}\label{e:wud}
\int_{\mathcal A_{r,R}^+}t^{1-2s}\nabla W_1\cdot\nabla U&=&\int_{\mathbf S^n_{R,+}} U t^{1-2s}\partial_r W_1-\int_{\mathbf S^n_{r,+}}U t^{1-2s}\partial_r W_1\\
&&+\kappa_{s}\int_{A_{r,R}}\frac{\Lambda_{n,s}}{|x|^{2s}}U(x,0)W_1(x,0).\notag
\end{eqnarray}
Therefore, by (\ref{e:uwd}) and (\ref{e:wud}), it holds that
\begin{eqnarray}\label{e:llwu}
&&\kappa_{s}\int_{A_{r,R}}\frac{\lambda-\Lambda_{n,s}}{|x|}W_1(x,0)U(x,0)+W_1(x,0) U^{\frac{n+2s}{n-2s}}(x,0)\\
&&=\int_{\mathbf S^n_{R,+}}\left(t^{1-2s} U \partial_r W_1-t^{1-2s}W_1 \partial_r U\right)-\int_{\mathbf S^n_{r,+}}\left(t^{1-2s}U \partial_r W_1-t^{1-2s}W_1 \partial_r U\right).\notag
\end{eqnarray}
Since $\lambda\ge \Lambda_{n,s}$ and both $W_1$ and $U$ are positive, the left side of (\ref{e:llwu}) is positive for all $r, R$.

On the other hand, by H\"older inequality, the right side of (\ref{e:llwu}) becomes
\begin{eqnarray*}
&&\left|\int_{\mathbf S^n_{R,+}}\left(t^{1-2s} U \partial_r W_1-t^{1-2s}W_1 \partial_r U\right)\right|\\
&&\le\left(\int_{\mathbf S^n_{R,+}}t^{1-2s}|U|^{2\gamma}\right)^{\frac{1}{2\gamma}}\left(\int_{\mathbf S^n_{R,+}}t^{1-2s}|\partial_r W_1|^{(2\gamma)'}\right)^{\frac{1}{(2\gamma)'}}\notag\\
&&\quad\quad\quad\quad\quad+\left(\int_{\mathbf S^n_{R,+}}t^{1-2s}|\partial_r U|^2\right)^{\frac{1}{2}}\left(\int_{\mathbf S^n_{R,+}}t^{1-2s}|W_1|^2\right)^{\frac{1}{2}},\notag
\end{eqnarray*}
where $\gamma$ is the constant given by Lemma \ref{l:ewss} and $1/(2\gamma)+1/(2\gamma)'=1$. Direct computations yield that
\begin{eqnarray*}
\left(\int_{\mathbf S^n_{R,+}}t^{1-2s}|\partial_r W_1|^{(2\gamma)'}\right)^{\frac{1}{(2\gamma)'}}&=&\left(\frac{n-2s}{2}\right)\left(\int_{\mathbf S^n_{R,+}}t^{1-2s}R^{\frac{2s-n-2}{2}(2\gamma)'}|\psi_1|^{(2\gamma)'}\right)^{\frac{1}{(2\gamma)'}}\notag\\
&=&\left(\frac{n-2s}{2}\right)R^{\frac{1}{2\gamma}}\left(\int_{\mathbf S^n_{1,+}}\theta_{n+1}^{1-2s}|\psi_1(\theta)|^{(2\gamma)'}\right)^{\frac{1}{(2\gamma)'}}
\end{eqnarray*}
and
\begin{eqnarray*}
\left(\int_{\mathbf S^n_{R,+}}t^{1-2s}|W_1|^2\right)^{\frac{1}{2}}&=&R^{\frac{1}{2}}\left(\int_{\mathbf S^n_{1,+}}\theta_{n+1}^{1-2s}\psi_1^2(\theta)d\theta\right)^{\frac{1}{2}}.
\end{eqnarray*}
Since $1\le(2\gamma)'<2$, it holds that
\begin{equation*}
\left(\int_{\mathbf S^n_{1,+}}\theta_{n+1}^{1-2s}|\psi_1(\theta)|^{(2\gamma)'}\right)^{\frac{1}{(2\gamma)'}}<\infty.
\end{equation*}
Then we have
\begin{eqnarray*}
&&\left|\int_{\mathbf S^n_{R,+}}\left(t^{1-2s} U \partial_r W_1-t^{1-2s}W_1 \partial_r U\right)\right|\\
&&\le\left(\frac{n-2s}{2}\right)\left(\int_{\mathbf S^n_{1,+}}\theta_{n+1}^{1-2s}|\psi_1(\theta)|^{(2\gamma)'}\right)^{\frac{1}{(2\gamma)'}}\left(R\int_{\mathbf S^n_{R,+}}t^{1-2s}|U|^{2\gamma}\right)^{\frac{1}{2\gamma}}\notag\\
&&\quad\quad\quad+\left(\int_{\mathbf S^n_{1,+}}\theta_{n+1}^{1-2s}\psi_1^2(\theta)d\theta\right)^{\frac{1}{2}}\left(R\int_{\mathbf S^n_{R,+}}t^{1-2s}|\partial_r U|^2\right)^{\frac{1}{2}}.\notag
\end{eqnarray*}
Note that $U\in \dot{H}^1(\mathbf R^{n+1}_+,t^{1-2s})$, by Lemma \ref{l:ewss}, $U\in L^{2\gamma}(\mathbf R^{n+1}_+,t^{1-2s})$. Therefore, there exists a sequence $R_k\to \infty$ ($k\to \infty$) such that
\begin{equation} \label{e:rk}
\left(R_k\int_{\mathbf S^n_{R_k,+}}t^{1-2s}|U|^{2\gamma}\right)^{\frac{1}{2\gamma}}+\left(R_k\int_{\mathbf S^n_{R_k,+}}t^{1-2s}|\partial_r U|^2\right)^{\frac{1}{2}}\to 0.
\end{equation}
Indeed, if otherwise, there is a constant $C>0$ such that
\begin{equation*}
R\int_{\mathbf S^n_{R,+}} t^{1-2s}|U|^{2\gamma}\ge C\quad\mbox{ and }\quad R\int_{\mathbf S^n_{R,+}}t^{1-2s} |\partial_r U|^2\ge C.
\end{equation*}
Then
\begin{equation*}
\int_{\mathbf S^n_{R,+}}t^{1-2s} |U|^{\frac{2(n+1)}{n-1}}\ge \frac{C}{R}\quad\mbox{ and }\quad \int_{\mathbf S^n_{R,+}}t^{1-2s}|\partial_r U|^2\ge \frac{C}{R}.
\end{equation*}
It is impossible since $R^{-1}$ is not integrable at $\infty$. So we have (\ref{e:rk}).

Similarly, there exists a sequence $r_k\to 0$ ($k\to\infty$) such that
\begin{equation*}
\left(r_k\int_{\mathbf S^n_{r_k,+}}t^{1-2s} |U|^{2\gamma}\right)^{\frac{1}{2\gamma}}+\left(r_k\int_{\mathbf S^n_{r_k,+}}t^{1-2s}|\partial_r U|^2\right)^{\frac{1}{2}}\to 0.
\end{equation*}
Thus we have
\begin{eqnarray*}
\int_{\mathbf S^n_{R_k,+}}\left( t^{1-2s}U \partial_{r} W_1-t^{1-2s}W_1 \partial_{r} U\right)-\int_{\mathbf S^n_{r_k,+}}\left(t^{1-2s}U \partial_r W_1-t^{1-2s}W_1 \partial_r U\right)\to 0.
\end{eqnarray*}
From (\ref{e:llwu}), it is impossible. This completes the proof.
\end{proof}


\section{Proof of Theorem \ref{t:existence-6}}\label{s:existence-d}

In this section, we prove Theorem \ref{t:existence-6}.

Let
\begin{equation}\label{e:minimization}
\beta(\lambda)=\inf_{U\in \dot{H}(\mathbf R^{n+1}_+,t^{1-2s})}\frac{Q(U)}{\|U(\cdot,0)\|_{2^*(s)}^2},
\end{equation}
where $Q(U)$ is given by (\ref{e:qud}).
By Sobolev inequality and trace inequality, we have
\begin{equation*}
\kappa_s S_{n,s}\|U(\cdot,0)\|_{2^*(s)}\le \int_{\mathbf R^{n+1}_+}t^{1-2s}|\nabla U|^2, \quad \forall \,U\in \dot{H}^{1}(\mathbf R^{n+1}_+,t^{1-2s}).
\end{equation*}
So $\beta(\lambda)$ is well-defined.
Recall that if $\lambda<\Lambda_{n,s}$, then $Q(U)\ge 0$. Thus $\beta(\lambda)> 0$.
\begin{lemma}\label{l:qukv}
If $\{U_k\}\subset \dot{H}^{1}(\mathbf R^{n+1}_+,t^{1-2s})$ weakly converges to $V\in \dot{H}^{1}(\mathbf R^{n+1}_+,t^{1-2s})$, then
\begin{equation}\label{e:qu0uk}
Q(U_k)=Q(V)+Q(U_k-V)+o(1).
\end{equation}
\end{lemma}
\begin{proof}
A direct computation yields
\begin{equation*}
Q(U_k)=Q(V+U_k-V)=Q(V)+Q(U_k-V)+2Q(V,U_k-V).
\end{equation*}
Note that
\begin{equation*}
Q(V,U_k-V)=\langle \mathcal L_Q V,U_k-V\rangle_{\dot{H}^{1}(\mathbf R^{n+1}_+,t^{1-2s})}\to 0,
\end{equation*}
then (\ref{e:qu0uk}) follows.
\end{proof}

Firstly, we consider a minimization sequence of (\ref{e:minimization}) with nontrivial weak limit.
\begin{lemma}\label{l:nontrivial-limit}
Let $0\le \lambda< \Lambda_{n,s}$. Suppose that $\{U_k\}\subset \dot{H}^{1}(\mathbf R^{n+1}_+,t^{1-2s})$ is a minimization sequence for (\ref{e:minimization}) and weakly converges to $V\not\equiv 0$, then $V$ is a minimum and $U_k\to V$ with respect to the norm topology in $\dot{H}^{1}(\mathbf R^{n+1}_+,t^{1-2s})$.
\end{lemma}
\begin{proof}
By trace inequality and compactness of the imbedding $\dot{H}^s(\mathbf R^n)\hookrightarrow L^p_{\rm loc}(\mathbf R^n)$, $1\le p< 2^*(s)$,  $U_k(\cdot,0)$ converges to $V(\cdot,0)$ almost everywhere. From Brezis-Lieb's result in \cite{BrezisLiebPAMS83}, it holds that
\begin{equation*}
\|U_k(\cdot,0)\|_{2^*(s)}^{2^*(s)}=\|V(\cdot,0)\|_{2^*(s)}^{2^*(s)}+\|U_k(\cdot,0)-V(\cdot,0)\|_{2^*(s)}^{2^*(s)}+o(1).
\end{equation*}
Since $Q(U_k)=\beta(\lambda)\|U_k(\cdot,0)\|_{2^*(s)}^2+o(1)$ and $Q(U_k-V)\ge \beta(\lambda)\|U_k(\cdot,0)-V(\cdot, 0)\|_{2^*(s)}^2$, we obtain that
\begin{eqnarray}\label{e:supukv}
\frac{Q(V)}{\|V(\cdot,0)\|_{2^*(s)}^2}
&=&\frac{Q(U_k)-Q(U_k-V)+o(1)}{\left(\|U_k(\cdot,0)\|_{2^*(s)}^{2^*(s)}-\|U_k(\cdot,0)-V(\cdot,0)\|_{2^*(s)}^{2^*(s)}+o(1)\right)^{\frac{2}{2^*(s)}}}\notag\\
&\le&\beta(\lambda)\frac{Q(U_k)-Q(U_k-V)+o(1)}{\left(Q^{\frac{2^*(s)}{2}}(U_k)-Q^{\frac{2^*(s)}{2}}(U_k-V)+o(1)\right)^{\frac{2}{2^*(s)}}}.
\end{eqnarray}
From Lemma \ref{l:qukv} and positivity of $Q(V)$, we get
\begin{eqnarray*}
0\le\limsup_k Q(U_k-V)=\limsup_k Q(U_k)-Q(V)<\limsup_k Q(U_k).
\end{eqnarray*}
We now argue by contradiction. If there exists a subsequence of $\{U_{k_i}\}$ such that $\displaystyle\lim_{i\to \infty} Q(U_{k_i}-V)>0$, then
\begin{equation*}
\lim_i\frac{Q(U_{k_i})-Q(U_{k_i}-V)+o(1)}{\left(Q^{\frac{2^*(s)}{2}}(U_{k_i})-Q^{\frac{2^*(s)}{2}}(U_{k_i}-V)+o(1)\right)^{\frac{2}{2^*(s)}}}<1.
\end{equation*}
This is impossible because of the definition of $\beta(\lambda)$ and (\ref{e:supukv}). Therefore, $\displaystyle\lim_{k\to \infty}Q(U_k-V)= 0$. Furthermore, recalling that $\lambda<\Lambda_{n,s}$ and (\ref{e:hardy-trace}), we have
\begin{eqnarray*}
Q(U_k-V)&=&\int_{\mathbf R^{n+1}_+}t^{1-2s}|\nabla (U_k-V)|^2-\kappa_s\lambda \int_{\mathbf R^n}\frac{(U_k(x,0)-V(x,0))^2}{|x|^{2s}}\notag\\
&\ge&\frac{\Lambda_{n,s}-\lambda}{\Lambda_{n,s}}\int_{\mathbf R^{n+1}_+}t^{1-2s}|\nabla (U_k-V)|^2\ge 0.
\end{eqnarray*}
That is, $U_k$ converges to $V$ in norm topology of $\dot{H}^{1}(\mathbf R^{n+1}_+,t^{1-2s})$. This completes the proof.
\end{proof}

Secondly, we investigate a minimization sequence of (\ref{e:minimization}) with trivial weak limit.
\begin{lemma}\label{l:rinvariant}
Let $U\in \dot{H}^{1}(\mathbf R^{n+1}_+,t^{1-2s})$,
\begin{equation*}
\mathcal R(U)(z)=R^{-\frac{n-2s}{2}}U\left(\frac{z}{R}\right),\quad R>0.
\end{equation*}
Then
\begin{equation*}
Q(\mathcal R(U))=Q(U),
\end{equation*}
\begin{equation*}
\|\mathcal R (U)(\cdot,0)\|_{2^*(s)}=\|U(\cdot, 0)\|_{2^*(s)}
\end{equation*}
and
\begin{equation*}
\|\mathcal R(U)(\cdot,0)\|_{\dot{H}^s(\mathbf R^n)}=\|U(\cdot,0)\|_{\dot{H}^s(\mathbf R^n)}.
\end{equation*}
\end{lemma}
\begin{proof}
The conclusions follow from some direct computations.
\end{proof}
By this lemma, we see that if $U$ is a minimization of (\ref{e:minimization}), then so is $\mathcal R(U)$.

\begin{lemma}\label{l:trivial-limit}
Let $0\le \lambda< \Lambda_{n,s}$. Suppose that $\{U_k\}\subset \dot{H}^{1}(\mathbf R^{n+1}_+,t^{1-2s})$ is a minimization sequence for (\ref{e:minimization}) and weakly converges to $0$,
then there is a sequence $\{\hat U_k\}$ associated with $\{U_k\}$ such that $\{\hat U_k\}$ is also a minimization sequence and
\begin{equation*}
\hat U_k\to V\not\equiv 0\mbox{ in }\dot{H}^{1}(\mathbf R^{n+1}_+,t^{1-2s}).
\end{equation*}
\end{lemma}
\begin{proof}
From the homogeneity of (\ref{e:minimization}), we may assume that $Q(U_k)\to \beta(\lambda)$ and $\|U_k(\cdot,0)\|_{2^*(s)}=1$.
Define
\begin{eqnarray}\label{e:qfnonl}
\mathcal Q(u,v)&:=&C_{n,s}\int_{\mathbf R^{2n}}\frac{(u(x)-u(y))(v(x)-v(y))}{|x-y|^{n+2s}}dxdy-\lambda\int_{\mathbf R^n}\frac{uv}{|x|^{2s}}\\
&=&\int_{\mathbf R^n}|\xi|^{2s}\mathcal F (u)(\xi)\mathcal F(v)(\xi)d\xi-\lambda\int_{\mathbf R^n}\frac{uv}{|x|^{2s}}\notag\\
&=&\langle u,v\rangle_{\dot{H}^s(\mathbf R^n)}-\lambda\int_{\mathbf R^n}\frac{uv}{|x|^{2s}}.\notag
\end{eqnarray}
Where $C_{n,s}=2^{2s-1}\pi^{-\frac{n}{2}}\Gamma(\frac{n+2s}{2})/|\Gamma(-s)|$ (see e.g. \cite{FLS:JAMS2007}). By Hardy-type inequality (\ref{e:hardy-ineq}), $\mathcal Q(\,\cdot\,,\,\cdot\,)$ are continuous in $\dot{H}^s(\mathbf R^n)\times\dot{H}^s(\mathbf R^n)$.
For simplicity, we also denote
\begin{equation*}
\mathcal Q(u):=\mathcal Q(u,u)=\langle u,u\rangle_{\dot{H}^s(\mathbf R^n)}-\lambda\int_{\mathbf R^n}\frac{u^2}{|x|^{2s}},\quad u\in \dot{H}^s(\mathbf R^n).
\end{equation*}

Let $u_k(x)=U_k(\cdot,0)$ and \begin{equation}\label{e:minimization1}
\tilde\beta(\lambda):=\inf_{u\in \dot{H}^s(\mathbf R^n)\setminus\{0\}}\frac{\mathcal Q(u)}{\|u\|_{2^*(s)}^2}.
\end{equation}
By a similar argument as in the proof of \cite[Theorem 1.5]{DiMoPeSc15}, we find a minimization sequence (by a decreasing rearrangement, using an improved Sobolev embedding inequality proved in \cite[Theorem 1.1]{PP:CVPDE14} and some proper rescaling of $u_k$ as in Lemma \ref{l:rinvariant}) $\tilde u_k$ of (\ref{e:minimization1}) such that
\begin{equation*}
\tilde u_k\to \tilde v\not\equiv 0 \mbox{ in }\dot{H}^s(\mathbf R^n).
\end{equation*}
Let $\tilde U_k(x,t)=(\mathcal P_s\ast\tilde u_k)(x,t)$ and $\tilde V(x,t)=(\mathcal P_s\ast\tilde v)(x,t)\not\equiv 0$. By (\ref{e:gtrace}), Lemma \ref{l:minimal-eu} and Lemma \ref{l:rinvariant},
we have $\tilde U_k\to \tilde V$ in $\dot{H}^1(\mathbf R^{n+1}_+,t^{1-2s})$ and $\tilde V$ is a minimizer of (\ref{e:minimization}). This completes the proof.
\end{proof}

\begin{proof}[Proof of Theorem \ref{t:existence-6}]
The result of Theorem \ref{t:existence-6} follows from Lemma \ref{l:nontrivial-limit} and Lemma \ref{l:trivial-limit}.
\end{proof}


\appendix

\section{Kelvin transformation}\label{s:kt}
The Kelvin transformation of $U$ is given by
\begin{equation*}
U_{X,\rho}(\xi):=\left(\frac{\rho}{|\xi-X|}\right)^{n-2s}U\left(X+\frac{\rho^2(\xi-X)}{|\xi-X|^2}\right),\quad \xi\in \mathbf R^{n+1}_+\setminus \{X\}.
\end{equation*}

\begin{lemma}\label{l:akte}
If $U$ satisfies (\ref{e:mail-extention}), then $U_{X,\rho}$ is a solution of the following problem
\begin{equation}\label{e:ktea}
\left\{\begin{array}{rl}
         {\rm div}(t^{1-2s}\nabla U_{X_0,\rho})(\xi)=0,\quad\quad{\rm in}\,\,\mathbf R^{n+1}_+, \\
         \begin{array}{rr}
        \displaystyle-\lim_{t\to 0^+}t^{1-2s}\partial_t U_{X_0,\rho}(\xi)=\kappa_s\left[\left(\frac{\rho}{|x-x_0|}\right)^{4s}\frac{\lambda}{|x_{\rho,x_0}|^{\alpha}}U_{X_0,\rho}(x_{\rho,x_0},0)\quad\quad\right.&\\
         +\left.\left(\dfrac{\rho}{|x-x_0|}\right)^{n+2s-p(n-2s)}(|U_{X_0,\rho}|^{p-1}U)(x_{\rho,x_0},0)\right], &\quad{\rm on}\,\,\mathbf R^n,
         \end{array}
       \end{array}
\right.
\end{equation}
where $\xi=(x,t)$, $X_0=(x_0,0)$ and
\begin{equation}\label{e:xxla}
x_{\rho,x_0}:=x_0+\frac{\rho^2(x-x_0)}{|x-x_0|^2}.
\end{equation}
\end{lemma}
\begin{proof}
The first equation of (\ref{e:ktea})  is a well-known fact. So we only give a detailed computation of the second equation of (\ref{e:ktea}).

Firstly, we compute $\partial_t U_{X_0,\rho}(x,t)$ as follows:
\begin{eqnarray*}
&&\frac{\partial}{\partial t}U_{X_0,\rho}(x,t)\\
&&=\frac{\partial}{\partial t}\left(\frac{\rho^{n-2s}}{|\xi-X_0|^{n-2s}}\right)U\left(X_0+\frac{\rho^2(\xi-X_0)}{|\xi-X_0|^2}\right)
+\frac{\rho^{n-2s}}{|\xi-X_0|^{n-2s}}\frac{\partial }{\partial t}U\left(X_0+\frac{\rho^2(\xi-X_0)}{|\xi-X_0|^2}\right)\notag\\
&&:=T_9+T_{10}.\notag
\end{eqnarray*}

A direct calculation yields that
\begin{equation*}
T_9=t\frac{(2s-n)\rho^{n-2s}}{|\xi-X_0|^{n+2-2s}}U\left(X_0+\frac{\rho^2(\xi-X_0)}{|\xi-X_0|^2}\right).
\end{equation*}

As for $Q_2$, we have
\begin{eqnarray*}
&&\frac{\partial}{\partial t}U\left(X_0+\frac{\rho^2(\xi-X_0)}{|\xi-X_0|^2}\right)\\
&&=\left(\frac{\rho^2}{|\xi-X_0|^2}\right)\left(\frac{\partial U}{\partial t}\right)\left(X_0+\frac{\rho^2(\xi-X_0)}{|\xi-X_0|^2}\right)\notag\\
&&-2t\rho^2\sum_{k=1}^{n+1}\left(\frac{\partial U}{\partial \xi_k}\right)\left(X_0+\frac{\rho^2(\xi-X_0)}{|\xi-X_0|^2}\right)\left(\frac{\xi_k-(X_0)_k}{|\xi-X_0|^4}\right).\notag
\end{eqnarray*}
Hence
\begin{eqnarray*}
T_{10}&=&\left(\frac{\rho^{n+2-2s}}{|\xi-X_0|^{n+2-2s}}\right)\left(\frac{\partial U}{\partial t}\right)\left(X_0+\frac{\rho^2(\xi-X_0)}{|\xi-X_0|^2}\right)\notag\\
&&-t\left(\frac{2\rho^{n+2-2s}}{|\xi-X_0|^{n+4-2s}}\right)\sum_{k=1}^{n+1}\left(\frac{\partial U}{\partial \xi_k}\right)\left(X_0+\frac{\rho^2(\xi-X_0)}{|\xi-X_0|^2}\right)\left(\xi_k-(X_0)_k\right).\notag
\end{eqnarray*}

Therefore, since $s\in (0,1)$,
\begin{eqnarray}\label{e:boundary}
&&-\lim_{t\to 0^+}t^{1-2s}\frac{\partial}{\partial t}U_{X_0,\rho}(x,t)\\
&&=-\lim_{t\to 0^+}\left(\frac{\rho^{n+2-2s}}{|\xi-X_0|^{n+2-2s}}\right)t^{1-2s}\left(\frac{\partial U}{\partial t}\right)\left(X_0+\frac{\rho^2(\xi-X_0)}{|\xi-X_0|^2}\right)\notag\\
&&=-\left(\frac{\rho^{n+2-2s}}{|x-x_0|^{n+2-2s}}\right)\lim_{t\to 0^+}t^{1-2s}\left(\frac{\partial U}{\partial t}\right)\left(x+\frac{\rho^2(x-x_0)}{|\xi-X_0|^2},\frac{\rho^2t}{|\xi-X_0|^2}\right)\notag\\
&&=-\left(\frac{\rho^{n+2s}}{|x-x_0|^{n+2s}}\right)\lim_{t\to 0^+}\left(\frac{\rho^2t}{|\xi-X_0|^2}\right)^{1-2s}\left(\frac{\partial U}{\partial t}\right)\left(x+\frac{\rho^2(x-x_0)}{|\xi-X_0|^2},\frac{\rho^2t}{|\xi-X_0|^2}\right)\notag.
\end{eqnarray}
From (\ref{e:mail-extention}), (\ref{e:boundary}) and the definition (\ref{e:xxla}) of $x_{\rho, x_0}$, it holds that
\begin{eqnarray*}
&&-\lim_{t\to 0^+}t^{1-2s}\frac{\partial}{\partial t}U_{X_0,\rho}(x,t)\\
&&=\kappa_s\left(\frac{\rho^{n+2s}}{|x-x_0|^{n+2s}}\right)\left(\frac{\lambda}{|x_{\rho,x_0}|^{\alpha}}U(x_{\rho,x_0},0)
+(|U|^{p-1}U)(x_{\rho,x_0},0)\right)\notag\\
&&=\kappa_s\left[\left(\frac{\rho}{|x-x_0|}\right)^{4s}\left(\frac{\lambda}{|x_{\rho,x_0}|^{\alpha}}\right)U_{X_0,\rho}(x_{\rho,x_0},0)\right.\notag\\
&&\quad\quad\quad+\left.\left(\frac{\rho}{|x-x_0|}\right)^{n+2s-p(n-2s)}(|U_{X_0,\rho}|^{p-1}U)(x_{\rho,x_0},0)\right].\notag
\end{eqnarray*}
Thus we have the second equation of (\ref{e:ktea}).
\end{proof}


\newcommand{\Toappear}{to appear in}

\bibliography{mrabbrev,cz_abbr2003-0,localbib}

\bibliographystyle{plain}
\end{document}